\documentclass[11pt,reqno]{amsart}
\usepackage{amssymb,amsmath,comment,amsthm,enumitem}
\catcode`\@=11 \@addtoreset{equation}{section}
\def\thesection{\arabic{section}}

\def\theequation{\thesection.\arabic{equation}}
\catcode`\@=12
\usepackage{colortbl}%
\usepackage{a4wide}
\usepackage{parskip}
\usepackage{mathtools}%%upgrade of amsmath
\usepackage{cite}

\usepackage[left=2.5cm,right=2.5cm,top=2.9cm,bottom=2cm]{geometry}

\usepackage[title]{appendix}

\usepackage[hidelinks]{hyperref}

\numberwithin{equation}{section}

\usepackage[all]{xy}
\catcode`\@=11
\def\theequation{\@arabic{\c@section}.\@arabic{\c@equation}}
\catcode`\@=12

\allowdisplaybreaks

\def\QED{\hfill {$\square$}\goodbreak \medskip}

\newtheorem{Theorem}{Theorem}[section]
\newtheorem{Lemma}[Theorem]{Lemma}

\newtheorem{Remark}{Remark}
\newtheorem{Definition}{Definition}[section]

\def\Xint#1{\mathchoice
	{\XXint\displaystyle\textstyle{#1}}%
	{\XXint\textstyle\scriptstyle{#1}}%
	{\XXint\scriptstyle\scriptscriptstyle{#1}}%
	{\XXint\scriptscriptstyle\scriptscriptstyle{#1}}%
	\!\int}
\def\XXint#1#2#3{{\setbox0=\hbox{$#1{#2#3}{\int}$ }
		\vcenter{\hbox{$#2#3$ }}\kern-.6\wd0}}

\newcommand{\Author}{Sun-Sig Byun, Kyeongbae Kim \and Deepak Kumar}

\begin{document}
\title[Nonlocal double phase with VMO coefficients]
	{Regularity results for a class of nonlocal double phase equations with VMO coefficients}
\address{Sun-Sig Byun: Department of Mathematical Sciences and Research Institute of Mathematics, Seoul National University, Seoul 08826, Korea}
\email{byun@snu.ac.kr}

\address{Kyeongbae Kim: Department of Mathematical Sciences, Seoul National University, Seoul 08826, Korea}
\email{kkba6611@snu.ac.kr}

\address{Deepak Kumar: Research Institute of Mathematics, Seoul National University, Seoul 08826, Korea}
\email{deepak.kr0894@gmail.com}

\thanks{S. Byun was supported by NRF-2022R1A2C1009312, K. Kim was supported by  NRF-2017R1A2B2003877 and D. Kumar was supported by NRF-2021R1A4A1027378.}
\author{\Author}
	
	\date{}
	\maketitle

\begin{abstract}
We study a class of nonlocal double phase problems with discontinuous coefficients. A local self-improving property and a higher H\"older continuity result for weak solutions to such problems are obtained under the assumptions that the associated coefficient functions are of type VMO (vanishing mean oscillation) and that the principal coefficient depends not only on the variables but also on the solution itself. 
 \medskip
		
\noindent \textbf{Key words:} Nonlocal double phase operators, Self-improving property, VMO coefficients,  higher H\"older regularity results.	
\medskip
		
 \noindent \textit{2010 Mathematics Subject Classification:} 35B65, 35J60, 35R11.
\end{abstract}

\section {Introduction}
In this article, we consider the following nonlocal problem:
	\begin{equation*}
		\mathcal{L}_{a(\cdot,u),b} \; u   = f \; \; \text{ in } \Omega, \tag{$\mathcal{P}$}\label{probM}
	\end{equation*}
 where $\Omega\subset\mathbb{R}^N$ is a bounded domain with $N\geq2$, $f\in L^\gamma_{\mathrm{loc}}(\Omega)$ with $\gamma>\max\{1,\frac{N}{ps}\}$ and the nonlocal operator $\mathcal{L}_{a(\cdot,),b}$ is defined as 
	\begin{align*}
		\mathcal{L}_{a(\cdot,u),b} \; u(x):=&2 \; {\mathrm P.V.} \int_{\mathbb R^N} a(x,y,u(x),u(y)) \frac{|u(x)-u(y)|^{p-2}(u(x)-u(y))}{|x-y|^{N+ps}}dy \nonumber \\
		&+2 \; {\mathrm P.V.} \int_{\mathbb R^N} b(x,y) \frac{|u(x)-u(y)|^{q-2}(u(x)-u(y))}{|x-y|^{N+qt}}dy, \quad\mbox{for }x\in\mathbb{R}^N,
	\end{align*}
with $1<p\leq q<\infty$,  $0<s, t<1$ and the kernel coefficients $a(\cdot,\cdot,\cdot,\cdot)$ and $b(\cdot,\cdot)$ are non-negative bounded functions. We will specify structural and regularity assumptions to be imposed on $a$ and $b$ later in the introduction.  
 Specifically, the nonlocal operator in this work is motivated by the double phase equations and quasi-linear equations for local cases, we refer to \cite{zhikov1,marce2,baroni,colombo,colombo2,defilip-Ming} and \cite{Ph,BPS} for each type of problem, respectively.\par 
 The primary objective of the paper is to establish a local self-improving property and a higher H\"older regularity result for weak solutions to a class of nonlocal double phase problems with possibly discontinuous coefficients and a leading kernel coefficient depending not only on the independent variables but also on the solution. Particularly, we assume that the kernel coefficients are of type VMO (Vanishing Mean Oscillation). We establish the self-improving property of local weak solutions that are locally in an appropriate fractional Sobolev space which extends the results of \cite{kuusi2} and \cite{SM}. To the best of our knowledge, there is no result to deal with a self-improving property of local weak solutions with nonzero boundary data. In this regard, our result gives a new path to treat a local weak solution with nonzero boundary data concerning a nonlocal Calder\'on-Zygmund theory as in \cite{Ni,NV}.  We also complement the results of higher H\"older continuity for local weak solutions of \cite{brasco} and \cite{JDSPDPP} to the case of more general kernel coefficients.\par

We now briefly mention some recent regularity results on nonlocal problems. In the case when the kernel coefficient $a(\cdot,\cdot)$ is independent of the solution and $b\equiv 0$; i.e., the fractional $p$-Laplacian type equations, Di Castro, Kuusi and Palatucci \cite{CKPf} proved the local H\"older regularity. Subsequently, for the case when the coefficients $a\equiv 1$ and $b\equiv 0$, Brasco, Lindgren and Schikorra \cite{brasco} obtained a higher H\"older regularity for the weak solutions to the problem \eqref{probM} for the super-quadratic case. For $p=2$, Nowak in \cite{NHH} established a similar result as in \cite{brasco} for problems involving irregular kernel coefficients. For additional regularity results of nonlocal equations, we refer to \cite{Fr,CSP,Cr,KH1,silvestre,caff-silv} and references therein. \par
Concerning the nonlocal double phase type problems, we refer to \cite{filippis} for H\"older regularity results for bounded viscosity solutions to the problem \eqref{probM} for the case when $a$ is independent of the solution and $qt\leq ps$. Later, Fang and Zhang in \cite{zhang}, and Byun, Ok and Song in \cite{BOS} obtained the H\"older continuity results for weak solutions to a similar problem when $tq\leq ps$ and $ps<tq$ (with the coefficient $b$ being H\"older continuous in the latter), respectively. Recently, Giacomoni, Kumar and Sreenadh in \cite{JDSPDPP} obtained higher H\"older continuity results with explicit H\"older exponent for weak solutions to the problem \eqref{probM} for the case $qt\leq ps$ and the coefficient $b$ being locally continuous only along the diagonals in $\Omega\times\Omega$. For some other regularity results of problems with non-standard growth nonlocal operators, we refer to \cite{BKO,chaker,defilip-ming-mix,GKS, GKS2}.\par
Regarding a self-improving property of weak solutions to the nonlocal equations, Kuusi, Mingione and Sire \cite{kuusi2} proved this property for fractional Laplacian type problems with linear growth by introducing a notion of dual pairs. Subsequently, Scott and Mengesha \cite{SM} extended this result to bounded weak solutions of \eqref{probM} with $a(x,y,u(x),u(y))=a(x,y)$ when $\frac{p-1}{p}\leq \frac{tq}{sp}\leq 1$. On the other hand, in \cite{ABES, Scf}, the authors employed different techniques such as functional analysis and harmonic analysis tools to obtain similar self-improving properties. Moreover,  we refer to \cite{MSY,FMSY,Ni,NV,BL} for Sobolev regularity results for nonlocal problems involving fractional Laplacian  type operators (or their nonlinear versions).
%and to \cite{KMS2} for a measure data problem.
\par
 
Motivated by the above discussion, in this article, we consider the problem \eqref{probM} with  the coefficient functions $a:\mathbb{R}^N\times\mathbb{R}^N\times\mathbb{R}\times\mathbb{R}\to\mathbb{R}$ and $b:\mathbb{R}^N\times\mathbb{R}^N\to \mathbb{R}$ satisfying the following:
\begin{itemize}
 \item[(A1)] the functions $a$ and $b$ are symmetric; that is, $a(x,y,z,w)=a(y,x,w,z)$ and $b(x,y)=b(y,x)$ for all $x,y\in\mathbb R^N$ and $z,w\in\mathbb R$;
 \item[(A2)] for all $x,y\in\mathbb{R}^N$ and $z,w\in \mathbb{R}$, there hold
		\begin{align}
			&0<\Lambda^{-1} \leq a(x,y,z,w)\leq \Lambda  \quad\mbox{and } \label{eqbdA}\\
			&0 \leq b(x,y)\leq \Lambda; \label{eqbdB}
		\end{align}
	\item[(A3)] the function $a$ is locally uniformly continuous in $\mathbb{R}^N\times\mathbb{R}^N\times\mathbb{R}\times\mathbb{R}$; that is, for any $M>0$, there is a non-decreasing function $\omega_{a,M}:[0,\infty)\rightarrow[0,\infty)$ with $\omega_{a,M}(0)=0$ and $\lim\limits_{t\downarrow 0}\omega_{a,M}(t)=0$ such that
		\begin{align*}
			|a(x,y,w,z)-a(x,y,w',z')|\leq\omega_{a,M}\Big(\frac{|w-w'|+|z-z'|}{2}\Big)
		\end{align*}
		for all $z,z',w,w'\in [-M,M]$ uniformly in $(x,y)\in\mathbb{R}^N\times\mathbb{R}^N$;
	\item[(A4)] the function $a(\cdot,\cdot,z,w)$ is in VMO on $\Omega\times\Omega$ locally uniformly in $z,w$ and the function $b(\cdot,\cdot)$ is in VMO on $\Omega\times\Omega$, in the sense of Definition \ref{defVMO} (see below).
	\end{itemize}
%%%%
%%%%
%%%%
%%%%
To prove our higher H\"older continuity result, we first obtain a self-improving identity for the weak solution to the problem \eqref{probM} much in the spirit of \cite{kuusi2} and \cite{SM}. It is worth mentioning that, unlike the previously mentioned works, for our case, solutions considered here are assumed to be locally bounded and locally in an appropriate fractional Sobolev space (see later in Section 2). However, this requires careful handling of the nonlocal tail terms. More precisely, we replace the notion of the standard nonlocal tail with a refined version as in \eqref{simv.n.tail} (see below) so that only the local behavior of the solution with respect to the fractional Sobolev space is taken into account. As a consequence, our self-improving identity holds for all $tq\leq ps$  without requiring any lower bound on the quantity $\frac{tq}{sp}$, unlike in \cite[(A2)]{SM}. Subsequently, we use a suitable approximation technique for $VMO$ coefficients to establish an appropriate comparison result, which finally yields the H\"older continuity estimates for weak solutions to the problem \eqref{probM}.\par
For the sake of completeness, we prove the existence of a weak solution to the problem \eqref{probM} with prescribed exterior data (see problem \eqref{genPrb}). For this, we use the theory of $M$-type operators (as described in \cite[Chapter II]{showalter}) defined on a suitable separable reflexive Banach space.
The main difficulty in this regard lies in the lack of monotonocity caused by the fact that the kernel coefficient $a$ depends on the solution. 
\noindent

Before introducing our main results, we give a definition of a local weak solution. See Section 2 for a precise definition of the terms involved.
\begin{Definition}[Local weak solution]
Let $f\in (\mathcal W(\Omega))^*$, then we say that $u\in\mathcal{W}_{\mathrm{loc}}(\Omega)\cap L^{p-1}_{ps}(\mathbb{R}^N)\cap L^{q-1}_{qt,b}(\mathbb{R}^N)$ is a local weak solution to the problem \eqref{probM} if for all $\phi\in \mathcal W(\Omega)$ with compact support contained in $\Omega$, there holds
	\begin{align}\label{eqWF}
	  &\int_{\mathbb R^N}\int_{\mathbb R^N}a(x,y,u(x),u(y)) \frac{[u(x)-u(y)]^{p-1}}{|x-y|^{N+ps}}(\phi(x)-\phi(y))dxdy \nonumber\\
	  &\quad+\int_{\mathbb R^N}\int_{\mathbb R^N}b(x,y) \frac{[u(x)-u(y)]^{q-1}}{|x-y|^{N+qt}}(\phi(x)-\phi(y))dxdy \nonumber \\
	   &=   \langle f,\phi\rangle_{\mathcal W, \mathcal W^*}.
		\end{align}
Local weak sub-solution (resp. super-solution) is defined similarly by  replacing the sign $``=$'' with $``\leq$ (resp. $\geq$)" in \eqref{eqWF} for all non-negative test functions.
\end{Definition}
We now introduce our main results.
The first one is the following local self-improving property of a weak solution to \eqref{probM}.
\begin{Theorem}[A priori estimate]\label{Main theorem1}
Suppose that $2\leq p\leq q\leq\frac{ps}{t}$  and that the assumptions (A1) and (A2) hold. Let $u\in W^{s,p}_{\mathrm{loc}}(\Omega)\cap L^{\infty}_{\mathrm{loc}}(\Omega)\cap L^{p-1}_{ps}(\mathbb{R}^{N})\cap L^{q-1}_{qt}(\mathbb{R}^{N})$ be a local weak solution to \eqref{probM} with the non-homogeneous term satisfying
	\begin{equation}
		\label{nonhomogeneous condition for SI}
		f\in L_{\mathrm{loc}}^{p_{*}+\delta_{0}}(\Omega) \quad\mbox{for some small } \delta_{0}>0,
	\end{equation}
where
\begin{equation*}
p_{*}=\frac{Np'}{N+sp'}\text{ if }sp<N\quad\text{and}\quad p_{*}=1 \text{ if }sp\geq N.
\end{equation*} 
Then, for all $\widetilde\Omega\Subset\Omega$, there is a constant $\delta=\delta(N,s,t,p,q,\Lambda,\delta_{0},\|u\|_{L^{\infty}(\widetilde\Omega)})>0$ such that $u\in W^{s+\frac{N\delta}{p(1+\delta)},p(1+\delta)}_{\mathrm{loc}}(\widetilde\Omega).$
In particular, there exists a constant $c$ depending only on $N,s,t,p,q,\Lambda,\delta_{0}$ and $\|u\|_{L^{\infty}(B_{2\rho_{0}}(x_{0}))}$ such that 
\begin{equation}
\label{self improving estimate}
\begin{aligned}
    &\left(\int_{B_{\frac{\rho_{0}}{2}}(x_{0})}\Xint-_{B_{\frac{\rho_{0}}{2}}(x_{0})}\left(\frac{|u(x)-u(y)|^{p}}{|x-y|^{N+ps}}\right)^{(1+\delta)}dxdy\right)^{\frac{p-1}{p(1+\delta)}}\\
    &\leq c\Bigg[\left(\int_{B_{2\rho_{0}}(x_{0})}\Xint-_{B_{2\rho_{0}}(x_{0})}\frac{|u(x)-u(y)|^{p}}{|x-y|^{N+ps}}dxdy\right)^{\frac{1}{p'}}+\rho_{0}^{s}\left(\Xint-_{B_{2\rho_{0}(x_{0})}}|f(x)|^{p_{*}+\delta_{0}}dx\right)^{\frac{1}{p_{*}+\delta_{0}}}\\
    &\qquad\qquad\qquad+\rho_{0}^{s-tq}T^{q-1}_{tq}(u;x_{0},2\rho_{0})+\rho_{0}^{-s(p-1)}T^{p-1}_{ps}(u;x_{0},2\rho_{0})+\rho_{0}^{-s(p-1)}\Bigg]
\end{aligned}
\end{equation}
whenever $B_{2\rho_{0}}(x_{0})\Subset\widetilde{\Omega}$ with $\rho_{0}\in(0,1]$.
\end{Theorem}
\begin{Remark}
We observe that if $u\in W^{s,p}(\mathbb{R}^{N})$, then 
\begin{equation*}
    u\in W^{s,p}_{\mathrm{loc}}(\Omega)\cap L^{p-1}_{ps}(\mathbb{R}^{N}),
\end{equation*}but the converse is not true. Therefore, as we pointed out earlier, our result generalizes the previous works in \cite{kuusi2} and \cite{SM}. On the other hand, if we consider the case for $b=0$, then it suffices to take a weak solution 
\begin{equation*}
    u\in W^{s,p}_{\mathrm{loc}}(\Omega)\cap L^{p-1}_{ps}(\mathbb{R}^{N}).
\end{equation*} 
\end{Remark}
We next describe the second main result which is the higher H\"older regularity.
\begin{Theorem}\label{thmmain}
  Suppose that $2\leq p\leq q< \min\{p^*_{s},ps/t\}$. Let the kernel coefficients satisfy the assumptions (A1) through (A4)  and let $u$ be a local weak solution to the problem \eqref{probM} such that $u\in  L^{q-1}_{qt}(\mathbb R^N)$. Then, $u\in C^{0,\alpha}_\mathrm{loc}(\Omega)$ for all $\alpha\in (0,{\Theta})$, where 
		\begin{align}\label{eqTheta}
		   {\Theta}:=\min\bigg\{ \frac{ps-N/\gamma}{p-1},\frac{qt}{q-1},1\bigg\}.
		\end{align}
	\end{Theorem}
Before ending the section, we mention the layout of the rest of the paper. Section 2 deals with some preliminaries related to the paper. Section 3 corresponds to the self-improving property and we prove Theorem \ref{Main theorem1}. Section 4 contains the proof of higher H\"older regularity result of Theorem \ref{thmmain}. Section 5 deals with the existence result for the problem \eqref{probM}. Finally, the appendix is devoted to some boundedness results.
\section{Preliminaries}
In this section, we give some notations and introduce related function spaces. Here, we will also recall some of the well known results.
\subsection{Notation} 
For $1<p<\infty$, we set $[\xi]^{p-1}=|\xi|^{p-2}\xi$, for all $\xi\in\mathbb R$. We abbreviate  
  $$(-\Delta)_{p,a(\cdot,u)}^{s}u(x):=  2{\mathrm P.V.}\int_{\mathbb R^N} a(x,y,u(x),u(y)) \frac{[u(x)-u(y)]^{p-1}}{|x-y|^{N+ps}}dy, \quad x\in\mathbb{R}^N$$
 and analogously $(-\Delta)^{t}_{q,b}$ is defined. The number $p'$ denotes the H\"older conjugate of $p$; that is, $1/p+1/p'=1$. Additionally,  for $1<p<\infty$ and $s\in (0,1)$, we define the Sobolev conjugate of $p$ by 
 \begin{align*}
     p^*_{s}:=\begin{cases}
     Np/(N-ps) &\mbox{ if } N>ps,\\
     \text{\r{p}} &\mbox{ if } N\leq ps,
     \end{cases} 
 \end{align*}
 where \r{p} is an arbitrarily large number.
  For $x_0\in\mathbb{R}^N$ and $v\in L^1(B_r(x_0))$, we set \[(v)_{r,x_0}:=\Xint-_{B_r(x_0)}v(x)dx=\frac{1}{|B_r(x_0)|}\int_{B_r(x_0)}v(x)dx\]
  and when the center is clear from the context, we will denote it as $(v)_r$.
 In several places, we will use 
	\begin{align*}
		d\mu_{1}(x,y):=\frac{dxdy}{|x-y|^{N+ps}} \quad\mbox{and}\quad	d\mu_{2}(x,y):=\frac{dxdy}{|x-y|^{N+qt}}.
	\end{align*} 
  For a Banach space $(X,\|\cdot\|)$, we denote its topological dual by $X^*$ and $\langle \cdot,\cdot \rangle_{X,X^*}$ denotes the duality pairing. 
The constant $c$ appearing in the proofs may vary line to line and is always greater than or equal to $1$. In particular, we write the relevant dependencies on parameters using parentheses; i.e., $c=c(n,s)$.  On the other hand, we write 
\begin{equation*}
	 \mathsf{data}\equiv \mathsf{data}(N,s,t,p,q,\Lambda).
	\end{equation*}

\subsection{Function spaces and definitions}
 For an open  set $\Omega\subset\mathbb R^N$, we define the space $\mathcal W_b(\Omega)$ as below: 
	\begin{equation*}
		\begin{aligned}
			\mathcal W_b(\Omega) :=\big\{u\in W^{s,p}(\Omega) \ : \ [u]_{W^{t,q}_b(\Omega)}+ \|u\|_{L^q(W_b,\:\Omega)} <\infty  \big \},
		\end{aligned}
	\end{equation*}
	equipped with the norm 
	\begin{align*}
		\|u\|_{\mathcal W_b(\Omega)}:= \|u\|_{L^p(\Omega)}+\|u\|_{L^q(W_b,\:\Omega)}+[u]_{W^{s,p}(\Omega)}+[u]_{W^{t,q}_b(\Omega)},
	\end{align*}
	where
	\begin{align}\label{eqw}
	\|u\|_{L^q(W_b,\:\Omega)}^q:=\int_{\Omega}W_b(x)|u(x)|^q dx \quad \mbox{with }	W_b(x):=\int_{\mathbb R^N\setminus\Omega}\frac{b(x,y)}{|x-y|^{N+qt}}dy
	\end{align}
 and 
 \begin{align*}
    [u]^q_{W^{t,q}_b(\Omega)}:=\int_{\Omega}\int_{\Omega} b(x,y)\frac{|u(x)-u(y)|^q}{|x-y|^{N+tq}}dxdy \quad\mbox{with } [u]_{W^{s,p}(\Omega)}=[u]_{W^{s,p}_1(\Omega)}.
 \end{align*}
Note that $C_c^\infty(\Omega)$ is obviously contained in $\mathcal{W}_b(\Omega)$. It is not difficult to verify that $\mathcal W_b(\Omega)$ is a uniformly convex Banach space.  Moreover, $\mathcal W_b(\Omega)$ is continuously embedded into $W^{s,p}(\Omega)$.\par
In what follows, the subscript $b$ from the definitions of  $\mathcal{W}_b(\Omega)$ and  $W_b$ will be suppressed if it has no relevance to the context. 
  We call a function $u\in \mathcal W_{\mathrm{loc}}(\Omega)$ if $u\in \mathcal W(\widetilde\Omega)$, for all $\widetilde\Omega\Subset\Omega$. 
Now we give definitions of tail space and kernel coefficient.
 \begin{Definition}%\label{defTail}
  Let $0<m,\alpha<\infty$ and $b(\cdot,\cdot)\in L^\infty(\mathbb R^N\times\mathbb R^N)$ be a non-negative function. Then, we define the tail space as below:
	\begin{align*}
	L^{m}_{\alpha,b}(\mathbb R^N) = \bigg\{ u:\mathbb{R}^{N}\to \mathbb{R}\text{ is measurable function } : \sup_{x\in\mathbb R^N}\int_{\mathbb R^N} b(x,y)\frac{|u(y)|^{m}}{(1+|y|)^{N+\alpha}}dy <\infty \bigg\}.
		\end{align*}
  And for $b(\cdot,\cdot)\equiv 1$, we denote it by $L^m_{\alpha}(\mathbb R^N)$. In particular, we write
  \begin{equation*}
      \|u\|_{L^{m}_{\alpha,b}(\mathbb{R}^{N})}=\left(\sup_{x\in\mathbb R^N}\int_{\mathbb R^N} b(x,y)\frac{|u(y)|^{m}}{(1+|y|)^{N+\alpha}}dy\right)^{\frac{1}{m}}.
  \end{equation*}
  For $0<\alpha<1<m<\infty$ and a measurable function $u:\mathbb R^N\to \mathbb R$, the nonlocal tail centered at $x_0\in\mathbb R^N$ with radius $R>0$ is defined as 
		\begin{align*}
		 T_{m\alpha,b}(u;x_0,R)=\Bigg(R^{m\alpha}\sup_{x\in\mathbb R^N}\int_{\mathbb R^N\setminus B_R(x_0)} b(x,y)\frac{|u(y)|^{m-1}}{|x_0-y|^{N+m\alpha}} dy\Bigg)^\frac{1}{m-1}.
		\end{align*}
 For $b(\cdot,\cdot)\equiv 1$, we follow the convention  $T_{m\alpha,1}(u;x_0,R)=T_{m\alpha}(u;x_0,R)$.
	\end{Definition}
\begin{Remark}
Using Minkowski's inequality, we check the following algebraic fact:
\begin{equation*}
    T_{m\alpha,b}(u+v;x_{0},R)\leq T_{m\alpha,b}(u;x_{0},R)+T_{m\alpha,b}(v;x_{0},R)
\end{equation*}
for any $u,v\in L^{m}_{\alpha,b}(\mathbb{R}^{N})$ with $m\geq2$. We often use this inequality when we deal with a tail estimate in Section 4.
\end{Remark}
\begin{Definition}[VMO functions]
\label{defVMO}
\begin{enumerate}
	\item For any $M>0$ and any ball $B_R\subset\Omega$, we say that the function $a$ is $(\delta_{M},R)-$vanishing in $B_R\times B_R$, if for all $x_{0},y_{0}\in B_R$ and $r\in(0,R]$ such that $B_{r}(x_{0}),B_{r}(y_{0})\subset B_R$,
		\begin{align*}
		 \Xint-_{B_{r}(x_{0})}\Xint-_{B_{r}(y_{0})}|a(x,y,w,z)-(a)_{r,x_{0},y_{0}}(w,z)|dxdy\leq\delta_{M}\text{, for all }w,z\in[-M,M],
			\end{align*}
	where $(a)_{r,x_{0},y_{0}}(w,z)=\Xint-_{B_{r}(x_{0})}\Xint-_{B_{r}(y_{0})}a(x,y,w,z)dxdy$.  
	\item We say that the function $a$ is in $VMO$ on $\Omega\times\Omega$ locally uniformly in $(w,z)$ if for any $M>0$ and $B_{\rho}(x),B_{\rho}(y)\subset\Omega$,
	\begin{align}\label{vmo modul}
	\nu_{a,M}(\rho):=\sup\limits_{|w|,|z|\leq M}\; \sup\limits_{0<r\leq \rho}\; \sup\limits_{x,y\in \Omega}\Xint-_{B_{r}(x)}\Xint-_{B_{r}(y)}|a(x',y',w,z)-(a)_{r,x,y}(w,z)|dx'dy'
 \end{align}
 	tends to $0$ as $\rho\downarrow 0$. \\
  Especially, if the function $a$ is independent of $(w,z)$, then we say that $a$ is in VMO on $\Omega\times\Omega$ and the $VMO$-modulus of $a$ is denoted by $\nu_{a}$:		\begin{align*}
		\nu_{a}(\rho)=\sup\limits_{0<r\leq \rho}\sup\limits_{x,y\in \Omega}\Xint-_{B_{r}(x)}\Xint-_{B_{r}(y)}|a(x',y')-(a)_{r,x,y}|dx'dy'.
			\end{align*}
		\end{enumerate}
	\end{Definition}

 We recall the following inequalities (see \cite{KKP}): 
\begin{itemize}
    \item 
    for $\ell\geq 2$, there exists a constant $c(\ell)>0$ such that 
	\begin{align}\label{eqmon}
		c(\ell)|\xi-\zeta|^\ell \leq ( [\xi]^{\ell-1}- [\zeta]^{\ell-1}) (\xi-\zeta)  \quad\mbox{for all }\xi,\zeta\in\mathbb R; 
	\end{align}
 \item for $\ell\geq 2$, there exists a constant $c=c(\ell)>0$ such that for all $\xi,\zeta\in\mathbb R,$
	\begin{align}\label{eqKKP2}
	   |[\xi-w]^{\ell-1} - [\zeta-w]^{\ell-1}|\leq c|\xi-\zeta|^{\ell-1}+c|\xi-\zeta||\xi-w|^{\ell-2}.
\end{align}
\end{itemize} 
Before ending this section, we mention the following iteration lemma, which will be used in the proof of Lemma \ref{key result of SI}. 
\begin{Lemma}(See \cite[Lemma 6.1]{GiD})
\label{technial lemma}
Let $\varphi$ be a bounded non-negative function in $[t_{1},t_{2}]$. For $t_{1}\leq r<\rho\leq t_{2}$, 
\begin{equation*}
    \varphi(r)\leq \eta\varphi(\rho)+\frac{M}{(\rho-r)^{\alpha}},
\end{equation*}
with $\eta\in(0,1)$, $M>0$ and $\alpha>0$. Then we have
\begin{equation*}
    \varphi(t_{1})\leq c\frac{M}{(t_{2}-t_{1})^{\alpha}},
\end{equation*}
for some constant $c=c(\eta,\alpha)$.
\end{Lemma}

\section{Self-improving properties}
Throughout this section, we assume that the local weak solution $u$ to the problem \eqref{probM} satisfies the following:
	\begin{equation*}
		%\label{solution space for SI}
		u\in W_{\mathrm{loc}}^{s,p}(\Omega)\cap L_{\mathrm{loc}}^{\infty}(\Omega)\cap L^{p-1}_{ps}(\mathbb{R}^{N})\cap L^{q-1}_{qt}(\mathbb{R}^{N})
	\end{equation*}
 with $2\leq p\leq q\leq\frac{sp}{t}$ and that \eqref{nonhomogeneous condition for SI} holds. We now fix 
 \begin{equation*}
     \widehat{\Omega}\Subset\widetilde{\Omega}\Subset\Omega.
 \end{equation*}In what follows, we write
 \begin{equation*}
     \mathsf{data_{1}}=\mathsf{data_{1}}\left(\mathsf{data},\|u\|_{L^{\infty}(\widetilde{\Omega})}\right).
 \end{equation*}
 For any $v\in L^{p-1}_{ps}(\mathbb{R}^{N})\cap L^{q-1}_{qt}(\mathbb{R}^{N})$, we denote
	\begin{equation*}
		T(v;x_{0},R)=\int_{\mathbb{R}^{N}\setminus B_{R}(x_{0})}\left(\frac{|v(y)|^{p-1}}{|x_{0}-y|^{N+sp}}+\|b\|_{L^{\infty}}\frac{|v(y)|^{q-1}}{|x_{0}-y|^{N+tq}}\right)dy,\quad B_{R}(x_{0})\subset\widetilde{\Omega}.
	\end{equation*}
For a unified approach to handle the forcing term, we set a non-negative number $\mathfrak{A}$ such that 
\begin{equation}
\label{condition of A}
\begin{cases}\mathfrak{A}=0\quad\text{if }sp<N,\\
\mathfrak{A}=\frac{1}{2}\min\{\delta_{0},1/p\}\quad\text{if }sp\geq N.
\end{cases}
\end{equation}
 Before going further,  we first give the following Caccioppoli-type estimate.
 \begin{Lemma}\label{Caccio estimate}
 Let $u$ be a local weak solution to \eqref{probM}. Let $B\equiv B_{R}(x_{0})\subset\Omega$ with $R\leq\frac{1}{8}$, and let $\psi\in C_{c}^{\infty}(B)$ be a cutoff function such that $0\leq\psi\leq 1$, $\psi\equiv1$ in $B_{r}(x_{0})$ and $|\nabla\psi|\leq\frac{4}{R-r}$ with $r\in(0,R)$. Then we have
	\begin{align*}	  &\int_{B}\Xint-_{B}\frac{|\psi^{\frac{q}{p}}(x)u(x)-\psi^{\frac{q}{p}}(y)u(y)|^{p}}{|x-y|^{N+sp}}+\int_{B}\Xint-_{B}b(x,y)\frac{|\psi(x)u(x)-\psi(y)u(y)|^{q}}{|x-y|^{N+tq}}dxdy\\
	  &\leq\frac{cR^{p(1-s)}}{(R-r)^{p}}\Xint-_{B}|u(x)|^{p}dx+\frac{cR^{q(1-t)}}{(R-r)^{q}}\Xint-_{B}|u(x)|^{q}dx+cR^{sp'}\left(\Xint-_{B}|f(x)|^{p_{*}+\mathfrak{A}}dx\right)^{\frac{p'}{p_{*}+\mathfrak{A}}}\\
	  &\quad+c\frac{R^{N+sp}}{(R-r)^{N+sp}}\int_{\mathbb{R}^{N}\setminus B}\frac{|u(y)|^{p-1}}{|x_{0}-y|^{N+sp}}dy\Xint-_{B}\psi^{q}(x)|u(x)|dx\\
	  &\quad+c\frac{R^{N+tq}}{(R-r)^{N+tq}}\int_{\mathbb{R}^{N}\setminus B}\|b\|_{L^{\infty}}\frac{|u(y)|^{q-1}}{|x_{0}-y|^{N+tq}}dy\Xint-_{B}\psi^{q}(x)|u(x)|dx,
		\end{align*}
	for some constant $c=c(\mathsf{data})$.
	\end{Lemma}
\begin{proof} 
Note that H\"older's inequality, Sobolev's embedding and Young's inequality imply that for any $\sigma>0$,
\begin{align}
\label{Holder ineq}
\int_{B}|fu\psi^{q}|dx&\leq \|f\psi^{\frac{q}{p'}}\|_{L^{p_{*}+\mathfrak{A}}(B)}\times\|u\psi^{\frac{q}{p}}\|_{L^{\left(p_{*}+\mathfrak{A}\right)'}(B)}\nonumber\\
    &\leq c(\mathsf{data})\|f\psi^{\frac{q}{p'}}\|_{L^{p_{*}+\mathfrak{A}}(B)}\times\left([u\psi^{\frac{q}{p}}]_{W^{s,p}(B)}+R^{-s}\|u\psi^{\frac{q}{p}}\|_{L^{p}(B)}\right)\nonumber\\
    &\leq \frac{c}{\sigma^{p-1}}\|f\|^{p'}_{L^{p_{*}+\mathfrak{A}}(B)}+\sigma\left([u\psi^{\frac{q}{p}}]^{p}_{W^{s,p}(B)}+R^{-sp}\|u\psi^{\frac{q}{p}}\|^{p}_{L^{p}(B)}\right).
\end{align}
In addition, since the kernel coefficient $a(\cdot,\cdot,\cdot,\cdot)$ satisfies the uniform ellipticity condition (A2)-\eqref{eqbdA}, we have the result of the lemma by following the proof as presented in \cite[Theorem 3.1]{SM} with \eqref{Holder ineq}.
\qed
\end{proof}

\textbf{A dual pair $(\mu,U)$.}
We now introduce a notion of dual pair $(\mu,U)$, which is an essential tool to obtain the self-improving property of a weak solution $u$ to \eqref{probM}.
Assume that $\epsilon$ is a sufficiently small positive number such that
	\begin{equation*}
		%\label{epsilon condition}    
		\epsilon\in\left(0,\min\left\{\frac{s}{p},1-s\right\}\right),
	\end{equation*} 
which will be determined later in Lemma \ref{key result of SI}.
 Let us define a measure $\mu$ in $\mathbb{R}^{2N}$ by
	\begin{equation*}
	 \mu(\mathcal{A})=\int_{\mathcal{A}}\frac{dxdy}{|x-y|^{N-\epsilon p}},\quad \mathcal{A}\subset\mathbb{R}^{2N}\mbox{ is a measurable subset.}
	\end{equation*} 
We write $\mathcal{B}(x_{0},R):=B_{R}(x_{0})\times B_{R}(x_{0})$ with $x_{0}\in\mathbb{R}^{N}$ and $R>0$. Then we observe 
some properties of the measure $\mu$ as below.
\begin{Lemma}
\label{measure mu}(See \cite[Theorem 3.1]{SM2}) 
%Let $\mathcal{B}(x_{0},R)=B_{R}(x_{0})\times B_{R}(x_{0})$ with $x_{0}\in\mathbb{R}^{n}$ and $R>0$. 
Let us write $k\mathcal{B}(x_{0},R)=\mathcal{B}(x_{0},kR)$ for $k>0$. Then we have
\begin{equation*}
    \frac{\mu(k\mathcal{B}(x_{0},R))}{\mu(\mathcal{B}(x_{0},R))}=k^{N+p\epsilon}\quad\mbox{and}\quad \mu(\mathcal{B}(x_{0},R))=\frac{cR^{N+p\epsilon}}{\epsilon},
\end{equation*}
where $c=c(N,p,\epsilon)$ satisfies $\frac{1}{C(N,p)}\leq c\leq C(N,p)$ for some constant $C(N,p)\geq 1$.	
\end{Lemma}
For $(x,y)\in\mathbb{R}^{N}\times \mathbb{R}^{N}$, we define some functions as below: 
\begin{equation}
	\label{Defintion of large}
\begin{cases}
\begin{aligned}
	 &B(x,y)=b(x,y)|x-y|^{(s-t)q+\epsilon(q-p)}, \quad U(x,y)=\frac{|u(x)-u(y)|}{|x-y|^{s+\epsilon}}, \\ &\mathbf{H}(x,y,U)=U^{p}+BU^{q},\quad G(x,y,U)=\mathbf{H}(x,y,U)^{\frac{1}{p'}},\\
	  &F(x,y)=
	  \begin{cases}
	  |f(x)|&\quad(x,y)\in\Omega\times\Omega,\\
	  0&\quad\mbox{otherwise}.
	  \end{cases}
\end{aligned}
\end{cases}
\end{equation}
Then we notice that 
\begin{equation*}
   G(x,y,U)\in L^{p'}_{\mathrm{loc}}(\Omega\times\Omega; d\mu)\quad\text{and}\quad F\in L_{\mathrm{loc}}^{p_{*}+\delta_{0}}(\Omega\times\Omega; d\mu).
\end{equation*}
For convenience of notations, we set 
  \begin{equation*}
	 m=\frac{Np+\epsilon p^{2}}{N+sp+\epsilon p},\quad \tau=s+\epsilon-\frac{\epsilon p}{m},\quad 
	 \alpha=\frac{m}{p}<1,\quad \theta=\frac{s-\epsilon(p-1)}{N+\epsilon p}
\end{equation*}
and 
\begin{equation*}
	 \beta_{i}=2^{i\left(-\frac{sp}{p-1}+(s+\epsilon)\right)},\quad \mbox{for non-negative integers } i.
\end{equation*}
Then we check directly the following:
\begin{equation}
\label{palphabeta}
    m\in(1,p),\quad p=\frac{Nm}{N-\tau m}\quad\mbox{and}\quad \sum_{i=0}^{\infty}\beta_{i}<\infty.
\end{equation}
 We now state the following fractional Sobolev inequality.
 \begin{Lemma}(See \cite[Lemma 4.2]{SM2})
 \label{poincare}
  Let $B_{R}(x_{0})\subset\Omega$ and let $u\in W^{s,p}(B_{R}(x_{0}))$.Then for all $\eta\in\left[1,p\right]$, we have
	\begin{equation*}
	\left(\Xint-_{B_{R}(x_{0})}|u(x)-(u)_{R,x_{0}}|^{\eta}dx\right)^{\frac{1}{\eta}}\leq \frac{c}{\epsilon^{\frac{1}{m}}}R^{s+\epsilon}\left(\Xint-_{\mathcal{B}(x_{0},R)}U^{m}d\mu\right)^{\frac{1}{m}},
		\end{equation*}
 for some constant $c=c(\mathsf{data})$.
	\end{Lemma}

With the aid of Caccioppoli-type inequality (see Lemma \ref{Caccio estimate}) and Lemma \ref{poincare}, we prove the following diagonal reverse H\"older-type inequality with a nonocal  tail.
 \begin{Lemma}\label{Reverse holder inequality}
 Let $u$ be a local weak solution to \eqref{probM}. Let  $0<R\leq\frac{1}{8}$ and choose a positive integer $l$ such that
		\begin{equation}
		\label{condition of l}
			x_{0}\in \widehat{\Omega},\quad 2^{l}R\leq2 \quad\mbox{and}\quad B_{2^lR}(x_0)\subset\widetilde{\Omega}.
		\end{equation}
 Then, for $B\equiv B_{R}(x_0)$, the following holds
	\begin{align}\label{reverse holder estimate}
	&\left(\Xint-_{\frac{1}{2}\mathcal{B}}G(x,y,U)^{p'}d\mu\right)^{\frac{1}{p'}} \nonumber\\
	&\leq c\frac{\sigma^{-(p-1)}}{\epsilon^{\frac{1}{p'\alpha}-\frac{1}{p'}}}\left(\Xint-_{\mathcal{B}}G(x,y,U)^{p'\alpha}d\mu\right)^{\frac{1}{p'\alpha}}+c\frac{\sigma}{\epsilon^{\frac{1}{p'\alpha}-\frac{1}{p'}}}\sum_{j=0}^{l}\beta_{j}^{p-1}\left(\Xint-_{2^{j}\mathcal{B}}G(x,y,U)^{p'\alpha}d\mu\right)^{\frac{1}{p'\alpha}} \nonumber\\
	& \quad+c\sigma\epsilon^{\frac{1}{p'}}\left[\epsilon\mu(\mathcal{B})\right]^{\theta}T\left(u-(u)_{2^{l}R,x_{0}};x_{0},2^{l}R\right)+\frac{c\left[\epsilon\mu(\mathcal{B})\right]^{\theta}}{\epsilon^{\frac{1}{p_{*}+\mathfrak{A}}-\frac{1}{p'}}}\left(\Xint-_{\mathcal{B}}F^{p_{*}+\mathfrak{A}}d\mu\right)^{\frac{1}{p_{*}+\mathfrak{A}}},
	\end{align}
 for some constant $c=c\left(\mathsf{data_{1}}\right)$ which is independent of $l$ and $\sigma\in(0,1)$.
	\end{Lemma}
	\begin{proof}
Let $l$ be a fixed positive number satisfying \eqref{condition of l}.
Using Lemma \ref{Caccio estimate} with $r=\frac{R}{2}$ and $\|u\|_{L^{\infty}(B_{R})}\leq c$, we deduce 
	\begin{align*}
	 &\underbrace{\int_{\frac{1}{2}B}\Xint-_{\frac{1}{2}B}\frac{|u(x)-u(y)|^{p}}{|x-y|^{N+sp}}dxdy+\int_{\frac{1}{2}B}\Xint-_{\frac{1}{2}B}b(x,y)\frac{|u(x)-u(y)|^{q}}{|x-y|^{N+tq}}dxdy}_{I_{1}}\\
	 &\quad\leq\underbrace{\frac{c}{R^{sp}}\Xint-_{B}|u(x)-(u)_{R,x_{0}}|^{p}dx}_{I_{2}}+\underbrace{cT(u-(u)_{R,x_{0}};x_{0},R)\Xint-_{B}\psi^{q}(x)|u(x)-(u)_{R,x_{0}}|dx}_{I_{3}} \\
	&\qquad+\underbrace{cR^{sp'}\left(\Xint-_{B}|f(x)|^{p_{*}+\mathfrak{A}}dx\right)^{\frac{p'}{p_{*}+\mathfrak{A}}}}_{I_{4}}.
		\end{align*}
 We estimate each $I_{i}$, for $i=1,2,3$ and $4$, to discover the reverse H\"older inequality \eqref{reverse holder estimate}.\\
\noindent
 \textbf{Estimate of $I_{1}$.} By \eqref{Defintion of large} and Lemma \ref{measure mu}, we observe that
	\begin{equation*}
		\Xint-_{\frac{1}{2}\mathcal{B}}G(x,y,U)^{p'}d\mu\leq c\frac{\epsilon}{R^{\epsilon p}}I_{1}.
	\end{equation*}
 \textbf{Estimate of $I_{2}$.} In light of Lemma \ref{poincare}, we have  
	 \begin{equation*}
		I_{2}\leq c\frac{R^{\epsilon p}}{\epsilon^{\frac{p}{m}}}\left(\Xint-_{\mathcal{B}}G(x,y,U)^{p'\alpha}d\mu\right)^{\frac{1}{\alpha}}.
	\end{equation*}
 \textbf{Estimate of $I_{3}$.} A simple calculation with \eqref{condition of l} gives us
\begin{align}\label{eq31} 
	 &\int_{\mathbb{R}^{N}\setminus B_{R}}\frac{|u(y)-(u)_{R,x_{0}}|^{q-1}}{|x_{0}-y|^{N+tq}}dy \nonumber\\
	 &\leq \sum_{i=0}^{l-1}\int_{{B_{2^{i+1}R}}\setminus B_{2^{i}R}}\frac{|u(y)-(u)_{R,x_{0}}|^{q-1}}{|x_{0}-y|^{N+tq}}dy + \int_{\mathbb{R}^{N}\setminus B_{2^{l}R}}\frac{|u(y)-(u)_{R,x_{0}}|^{q-1}}{|x_{0}-y|^{N+tq}}dy \nonumber\\
	 &\eqqcolon\sum_{i=0}^{l-1}I_{3,i}+I_{3,l}.
\end{align}
We first note from the facts $\|u\|_{L^{\infty}(B_{2})}\leq c$ and $p\leq q$ that
\begin{equation*}
\begin{aligned}
I_{3,i}^{\frac{1}{p-1}}& \leq c\left((2^{i}R)^{-qt}\Xint-_{B_{2^{i+1}R}}|u-(u)_{B_{R}}|^{p-1}dy\right)^\frac{1}{p-1}\\
    &\leq c(2^{i}R)^{-\frac{qt}{p-1}}\left[\left(\Xint-_{B_{2^{i+1}R}}|u-(u)_{B_{2^{i+1}R}}|^{p-1}dy\right)^\frac{1}{p-1}+\sum_{j=0}^{k}\left|(u)_{B_{2^{j+1}R}}-(u)_{B_{2^{j}R}}\right|\right]\\
    &\leq c(2^{i}R)^{-\frac{qt}{p-1}}\sum_{j=1}^{i+1}\left(\Xint-_{B_{2^{j+1}R}}|u-(u)_{B_{2^{j}R}}|^{p-1}dy\right)^\frac{1}{p-1}\\
    &\leq c(2^{i}R)^{-\frac{sp}{p-1}}\sum_{j=1}^{i+1}\left(\Xint-_{B_{2^{j+1}R}}|u-(u)_{B_{2^{j}R}}|^{p-1}dy\right)^\frac{1}{p-1},
\end{aligned}
\end{equation*}
where we have used  the relations $2^{l}R\leq 2$ and $tq\leq sp$ in the last inequality. Then, by Lemma \ref{poincare}, we obtain
\begin{equation*}
    I_{3,i}^{\frac{1}{p-1}} \leq c(2^{i}R)^{-\frac{sp}{p-1}}\sum_{j=1}^{i+1}\frac{2^{j(s+\epsilon)}R^{s+\epsilon}}{\epsilon^{\frac{1}{m}}}\left(\Xint-_{2^{j}\mathcal{B}}U^{m}d\mu\right)^{\frac{1}{m}}.
\end{equation*}
We now employ the following Minkowski's inequality
\begin{equation*}
    \left(\sum_{j=0}^{i}\left(I_{3,i}^{\frac{1}{p-1}}\right)^{p-1}\right)^{\frac{1}{p-1}}\leq \sum_{j=0}^{i}I_{3,i}^{\frac{1}{p-1}}
\end{equation*}
and Fubini's theorem to deduce that
\begin{equation*}
\begin{aligned}
\left(\sum_{i=0}^{l-1}I_{3,i}\right)^{\frac{1}{p-1}}&\leq c\sum_{i=0}^{l-1}(2^{i}R)^{-\frac{sp}{p-1}}\sum_{j=1}^{i+1}\frac{2^{j(s+\epsilon)}R^{s+\epsilon}}{\epsilon^{\frac{1}{m}}}\left(\Xint-_{2^{j}\mathcal{B}}U^{m}d\mu\right)^{\frac{1}{m}}\\
&\leq c\sum_{j=1}^{l}\sum_{i=j-1}^{l-1}\left(2^{i}R\right)^{-\frac{sp}{p-1}}\frac{2^{j(s+\epsilon)}R^{s+\epsilon}}{\epsilon^{\frac{1}{m}}}\left(\Xint-_{2^{j}\mathcal{B}}U^{m}d\mu\right)^{\frac{1}{m}}\\
	 &\leq c\sum_{j=1}^{l}R^{-\frac{sp}{p-1}+s+\epsilon}\frac{\beta_{j}}{\epsilon^{\frac{1}{m}}}\left(\Xint-_{2^{j}\mathcal{B}}U^{m}d\mu\right)^{\frac{1}{m}}.
\end{aligned}
\end{equation*}
By using Minkowski's inequality once again, we next note that
\begin{equation}\label{simv.eq.3.l}
\begin{aligned}
I_{3,l}^{\frac{1}{q-1}}&\leq \left(\int_{\mathbb{R}^{n}\setminus B_{2^{l}R}}\frac{|u-(u)_{B_{2^{l}R}}|^{q-1}}{|y|^{n+tq}}dy\right)^{\frac{1}{q-1}}\\
    &\quad+\sum_{j=0}^{l-1}\left(\int_{\mathbb{R}^{n}\setminus B_{2^{i}R}}\frac{|(u)_{B_{2^{j+1}R}}-(u)_{B_{2^{j}R}}|^{q-1}}{|y|^{n+tq}}dy\right)^{\frac{1}{q-1}}.
\end{aligned}
\end{equation}
 We further estimate the second term on the right-hand side of \eqref{simv.eq.3.l} by means of Lemma \ref{poincare} as below
\begin{equation*}
\begin{aligned}
    &\sum_{j=0}^{l-1}\left(\int_{\mathbb{R}^{n}\setminus B_{2^{l}R}}\frac{|(u)_{B_{2^{j+1}R}}-(u)_{B_{2^{j}R}}|^{q-1}}{|y|^{n+tq}}dy\right)^{\frac{1}{q-1}}\\
    &\leq c\sum_{j=0}^{l-1}\left(\int_{\mathbb{R}^{n}\setminus B_{2^{l}R}}\frac{|(u)_{B_{2^{j+1}R}}-(u)_{B_{2^{j}R}}|^{p-1}}{|y|^{n+tq}}dy\right)^{\frac{1}{q-1}}\\
    &\leq c\sum_{j=0}^{l-1}(2^{l}R)^{-\frac{qt}{q-1}}\left|(u)_{B_{2^{j+1}R}}-(u)_{B_{2^{j}R}}\right|^{\frac{p-1}{q-1}}\\
    %&\leq c\sum_{j=1}^{l-1}(2^{l}R)^{-\frac{sp}{q-1}}\left(\frac{2^{j(s+\epsilon)}R^{s+\epsilon}}{\epsilon^{\frac{1}{m}}}\left(\Xint-_{2^{j}\mathcal{B}}U^{m}d\mu\right)^{\frac{1}{m}}\right)^{\frac{p-1}{q-1}}\\
    &\leq c\sum_{j=1}^{l-1}\left((2^{l}R)^{-\frac{sp}{p-1}}\frac{2^{j(s+\epsilon)}R^{s+\epsilon}}{\epsilon^{\frac{1}{m}}}\left(\Xint-_{2^{j}\mathcal{B}}U^{m}d\mu\right)^{\frac{1}{m}}\right)^{\frac{p-1}{q-1}}.
\end{aligned}
\end{equation*}
 We next claim that
\begin{equation}
\label{qptail}
    I_{3,l}^{\frac{1}{p-1}}\leq cT(u-(u)_{2^{l}R,x_{0}};x_{0},2^{l}R)^{\frac{1}{p-1}}+c\sum_{j=1}^{l-1}(2^{j}R)^{-\frac{sp}{p-1}}\frac{2^{j(s+\epsilon)}R^{s+\epsilon}}{\epsilon^{\frac{1}{m}}}\left(\Xint-_{2^{j}\mathcal{B}}U^{m}d\mu\right)^{\frac{1}{m}}.
\end{equation}
Indeed, if $p=q$, it is a direct computation. We now assume that $p<q$. Then, by H\"older's inequality, we have
\begin{equation*}
\begin{aligned}
    I_{3,l}^{\frac{1}{p-1}}&\leq \left(I_{3,l}^{\frac{1}{q-1}}\right)^{\frac{q-1}{p-1}}\\
    &\leq cT(u-(u)_{2^{l}R,x_{0}};x_{0},2^{l}R)^{\frac{1}{p-1}}\\
    &\quad+c\left(\sum_{j=1}^{l-1}\left(2^{l-j}\right)^{\frac{-sp}{q-1}}\left((2^{j}R)^{-\frac{sp}{p-1}}\frac{2^{j(s+\epsilon)}R^{s+\epsilon}}{\epsilon^{\frac{1}{m}}}\left(\Xint-_{2^{j}\mathcal{B}}U^{m}d\mu\right)^{\frac{1}{m}}\right)^{\frac{p-1}{q-1}}\right)^{\frac{q-1}{p-1}}\\
    &\leq cT(u-(u)_{2^{l}R,x_{0}};x_{0},2^{l}R)^{\frac{1}{p-1}}\\
    &\quad+c\sum_{j=1}^{l-1}(2^{j}R)^{-\frac{sp}{p-1}}\frac{2^{j(s+\epsilon)}R^{s+\epsilon}}{\epsilon^{\frac{1}{m}}}\left(\Xint-_{2^{j}\mathcal{B}}U^{m}d\mu\right)^{\frac{1}{m}},
\end{aligned}
\end{equation*}
which proves \eqref{qptail}. Consequently,
\begin{equation}\label{simv.tl.q}
\begin{aligned}
&\int_{\mathbb{R}^{N}\setminus B_{R}}\frac{|u(y)-(u)_{R,x_{0}}|^{q-1}}{|x_{0}-y|^{N+tq}}dy \\
&\leq \left(\sum_{i=0}^{l-1}I_{3,i}^{\frac{1}{p-1}}+I_{3,l}^{\frac{1}{p-1}}\right)^{p-1} \\
&\leq cT(u-(u)_{2^{l}R,x_{0}};x_{0},2^{l}R)+c\frac{R^{\epsilon p-(s+\epsilon)}}{\epsilon^{\frac{p-1}{m}}}\sum_{j=0}^{l}\left(\beta_{j}\left(\Xint-_{2^{j}\mathcal{B}}U^{m} d\mu\right)^{\frac{1}{m}}\right)^{p-1},
\end{aligned}
\end{equation}
where we have used the following algebraic inequality
\begin{equation*}
\begin{aligned}
&\left(\sum_{j=1}^{l}R^{-\frac{sp}{p-1}+s+\epsilon}\frac{\beta_{j}}{\epsilon^{\frac{1}{m}}}\left(\Xint-_{2^{j}\mathcal{B}}U^{m}d\mu\right)^{\frac{1}{m}}\right)^{p-1}\\
&\leq \left(\sum_{j=1}^{l}R^{-\frac{sp}+(s+\epsilon)(p-1)}\frac{\beta_{j}}{\epsilon^{\frac{p-1}{m}}}\left(\Xint-_{2^{j}\mathcal{B}}U^{m}d\mu\right)^{\frac{p-1}{m}}\right)\left(\sum_{j=1}^{l}\beta_{j}\right)^{p-2}\\
&\leq c\frac{R^{\epsilon p-(s+\epsilon)}}{\epsilon^{\frac{p-1}{m}}}\sum_{j=0}^{l}\left(\beta_{j}\left(\Xint-_{2^{j}\mathcal{B}}U^{m} d\mu\right)^{\frac{1}{m}}\right)^{p-1}.
\end{aligned}
\end{equation*}
Similarly, we estimate 
\begin{equation}\label{simv.tl.p}
\begin{aligned}
\int_{\mathbb{R}^{N}\setminus B_{R}}\frac{|u(y)-(u)_{R,x_{0}}|^{p-1}}{|x_{0}-y|^{N+sp}}dy&\leq  cT(u-(u)_{2^{l}R,x_{0}};x_{0},2^{l}R) \\
&\quad+c\frac{R^{\epsilon p-(s+\epsilon)}}{\epsilon^{\frac{p-1}{m}}}\sum_{j=0}^{l}\left(\beta_{j}\left(\Xint-_{2^{j}\mathcal{B}}U^{m} d\mu\right)^{\frac{1}{m}}\right)^{p-1}.
\end{aligned}
\end{equation}
Coupling \eqref{simv.tl.q} and \eqref{simv.tl.p}, we get
	\begin{equation}\label{tail and mean}  
	 \begin{aligned}
	 T(u-(u)_{R,x_{0}};x_{0},R)&\leq  c\frac{R^{\epsilon p-(s+\epsilon)}}{\epsilon^{\frac{p-1}{m}}}\sum_{j=0}^{l}\left(\beta_{j}\left(\Xint-_{2^{j}\mathcal{B}}U^{m} d\mu\right)^{\frac{1}{m}}\right)^{p-1}\\
	 &\quad+cT(u-(u)_{2^{l}R,x_{0}};x_{0},2^{l}R).
			\end{aligned}
	\end{equation}
 On the other hand, we have
	\begin{equation}
		\label{uminumin}
	 \Xint-_{B}\psi^{q}(x)|u(x)-(u)_{R,x_{0}}|dx\leq c\frac{R^{s+\epsilon}}{\epsilon^{\frac{1}{m}}}\left(\Xint-_{\mathcal{B}}U^{m} d\mu\right)^{\frac{1}{m}}.
		\end{equation}
Consequently, using Young's inequality with \eqref{tail and mean} and \eqref{uminumin}, we obtain
	\begin{equation*}
	 \begin{aligned}
		I_{3}&\leq c\frac{\sigma^{p'}R^{\epsilon p}}{\epsilon^{\frac{p}{m}}}\left(\sum_{j=0}^{l}\beta_{j}\left(\Xint-_{2^{j}\mathcal{B}}U^{m} d\mu\right)^{\frac{p-1}{m}}\right)^{p'}+c\sigma^{p'}R^{sp'}T(u-(u)_{2^{l}R,x_{0}};x_{0},2^{l}R)^{p'}\\
		&\quad+c\frac{\sigma^{-p} R^{\epsilon p}}{\epsilon^{\frac{p}{m}}}\left(\Xint-_{\mathcal{B}}U^{m} d\mu\right)^{\frac{p}{m}}\\
		&\leq c\frac{\sigma^{p'}R^{\epsilon p}}{\epsilon^{\frac{p}{m}}}\left(\sum_{j=0}^{l}\beta_{j}\left(\Xint-_{2^{j}\mathcal{B}}G^{p'\alpha} d\mu\right)^{\frac{1}{p'\alpha}}\right)^{p'}+c\sigma^{p'}R^{sp'}T(u-(u)_{2^{l}R,x_{0}};x_{0},2^{l}R)^{p'}\\
		&\quad+c\frac{\sigma^{-p} R^{\epsilon p}}{\epsilon^{\frac{p}{m}}}\left(\Xint-_{\mathcal{B}}G^{p'\alpha} d\mu\right)^{\frac{p}{m}}.
			\end{aligned}
		\end{equation*}
 \textbf{Estimate of $I_{4}$.} Recalling the definition of $F$ from \eqref{Defintion of large}, we get
	 \begin{equation*}
	  I_{4}\leq c\frac{R^{sp'}}{\epsilon^{\frac{p'}{p_{*}+\mathfrak{A}}}}\left(\Xint-_{\mathcal{B}}F^{p_{*}+\mathfrak{A}} d\mu\right)^{\frac{p'}{p_{*}+\mathfrak{A}}}.
		\end{equation*}
Eventually, we combine the estimates of $I_{1},I_{2},I_{3}$ and $I_{4}$ to discover that
\begin{equation*}
\begin{aligned}
\Xint-_{\frac{1}{2}\mathcal{B}}G(x,y,U)^{p'}d\mu&\leq \frac{c\sigma^{-p}}{\epsilon^{\frac{p}{m}-1}}\left(\Xint-_{\mathcal{B}}G(x,y,U)^{p'\alpha}d\mu\right)^{\frac{1}{\alpha}}+\frac{c\sigma^{p'}}{\epsilon^{\frac{p}{m}-1}}\left(\sum_{j=0}^{l}\beta_{j}^{p-1}\left(\Xint-_{2^{j}\mathcal{B}}G^{p'\alpha} d\mu\right)^{\frac{1}{p'\alpha}}\right)^{p'}\\
&+c\epsilon\sigma^{p'}R^{sp'-\epsilon p}T(u-(u)_{2^{l}R,x_{0}};x_{0},2^{l}R)^{p'}+c\frac{R^{sp'-\epsilon p}}{\epsilon^{\frac{p'}{p_{*}+\mathfrak{A}}-1}}\left(\Xint-_{\mathcal{B}}F^{p_{*}+\mathfrak{A}} d\mu\right)^{\frac{p'}{p_{*}+\mathfrak{A}}},
\end{aligned}
\end{equation*}
which implies 
 \eqref{reverse holder estimate}.\QED
	\end{proof}
%%%	
 Now we are ready to give and prove a level set estimate for $G$ in $\mathcal{B}(x_{0},2\rho_{0})\subset\widetilde{\Omega}\times\widetilde{\Omega}$  with $\rho_{0}\leq1$ and $x_{0}\in \widehat{\Omega}$. 
 First, we introduce a few more functionals.
 For every $\mathcal{B}(x,R)\subset\mathcal{B}(x_{0},2\rho_{0})$, we define
\begin{align*}
&\Upsilon(x,R):=\left(\Xint-_{\mathcal{B}(x,R)}F^{p_{*}+\mathfrak{A}+\delta_{f}}d\mu\right)^{\frac{1}{p_{*}+\mathfrak{A}+\delta_{f}}}, \end{align*}
where $\delta_{f}\in\left(0,\frac{\delta_{0}}{2}\right]$ to be determined later in Lemma \ref{key result of SI}, and
	\begin{equation}\label{simv.n.tail}
	  \begin{aligned}
	  \mathrm{Tail}(x,R)&:=\sum_{i=0}^{l}\beta_{i}^{p-1}\Big(\Xint-_{2^{i}\mathcal{B}(x,R)}G(x,y,U)^{p'\alpha}d\mu\Big)^{\frac{1}{p'\alpha}} \\
	  &\quad+\epsilon^{\frac{1}{p'}}\left[\epsilon\mu(\mathcal{B}(x,R))\right]^{\theta}T(u-(u)_{2^{l}R,x};x,2^{l}R),
		\end{aligned}
	\end{equation}
 for some positive integer $l$ such that
	\begin{equation*}
		\frac{\rho_{0}}{2}\leq 2^{l}R<\rho_{0}.
	\end{equation*}
 We also define 
	\begin{equation*}
		\Psi_{M}(x,R):= \left(\Xint-_{\mathcal{B}(x,R)}G^{p'}d\mu\right)^{\frac{1}{p'}}+M\frac{\left[\mu(\mathcal{B}(x,R))\right]^{\theta}}{\epsilon^{\frac{1}{p_{*}+\mathfrak{A}}-\frac{1}{p'}}}\left(\Xint-_{\mathcal{B}(x,R)}F^{p_{*}+\mathfrak{A}}d\mu\right)^{\frac{1}{p_{*}+\mathfrak{A}}},
	\end{equation*}
 where $M\geq1$ will be chosen later in \eqref{condition of sigma and M}. Now we set 
	\begin{equation*}
	 \Xi(x,R):=\Upsilon(x,R)+\mathrm{Tail}(x,R)+\Psi_{M}(x,R).
	\end{equation*}
 In particular, we denote 
  \begin{equation}
  \label{theta0}
	\Xi_{0}:=\Upsilon(x_{0},2\rho_{0})+\Psi_{1}(x_{0},2\rho_{0})+T(u-(u)_{2\rho_{0},x_{0}};x_{0},2\rho_{0}).
 \end{equation}
 For convenience, we write
\begin{equation*}
\theta_{u}=\frac{(p+1)(1-\alpha)}{\alpha}
    ,\quad \theta_{f}=(p_{*}+\mathfrak{A}+\delta_{f})\left(\frac{\left(p_{*}+\mathfrak{A}\right)\theta}{1-\left(p_{*}+\mathfrak{A}\right)\theta}\right)\quad\mbox{and}\quad \Tilde{\theta}_{f}=\frac{\left(p_{*}+\mathfrak{A}\right)(1+\theta\delta_{f})}{1-\left(p_{*}+\mathfrak{A}\right)\theta}.
\end{equation*}	
In this setting, we now describe an integral estimate of $G$ on super-level sets.
 \begin{Lemma}\label{level set estimate lemma}
 Suppose that $u$ is a local weak solution to \eqref{probM}. Take $B\equiv B_{2\rho_{0}}(x_{0})\subset\widetilde{\Omega}$ with $0<\rho_{0}\leq1$ and $x_{0}\in \widehat{\Omega}$. Let $\frac{\rho_{0}}{2}\leq r<\rho\leq\rho_{0}$.
 Then there exist constants $c_{\alpha}=c_{\alpha}(\mathsf{data_{1}})\geq1$,  $c_{f}=c_{f}(\mathsf{data_{1}},\epsilon)\geq1$ and $\kappa_{f}=\kappa_{f}(\mathsf{data_{1}},\epsilon)\in(0,1)$ such that the inequality
	\begin{equation}\label{level set estimate}
	 \frac{1}{\lambda^{p'}}\int_{\mathcal{B}(x_{0},r)\cap\{G>\lambda\}}G^{p'}d\mu\leq \frac{c_{\alpha}}{\epsilon^{\theta_{u}}\lambda^{p'\alpha}}\int_{\mathcal{B}(x_{0},\rho)\cap\{G>\lambda\}}G^{p'\alpha}d\mu+\frac{c_{f}\lambda_{0}^{\theta_{f}}}{\lambda^{\Tilde{\theta}_{f}}}\int_{\mathcal{B}(x_{0},\rho)\cap\{F>\kappa_{f}\lambda\}}F^{p_{*}+\mathfrak{A}}d\mu,
		\end{equation}
 holds whenever $\lambda\geq\lambda_{0}$, where
	\begin{equation}
	\label{def of lambda0}
	 \lambda_{0}:=\frac{c}{\epsilon^{\frac{1}{p'\alpha}}}\left(\frac{\rho_{0}}{\rho-r}\right)^{2N+p}\Xi_{0},
		\end{equation}
	for some constant $c=c\left(\mathsf{data_{1}}\right)$.
	\end{Lemma}
\begin{proof}
Let $\kappa\in(0,1)$ be a parameter which will be determined later in  \eqref{nondiagonal level set estimate}. 
	Define
	\begin{equation}
	\label{defn of lambda1}
		\lambda_{1}:=\frac{1}{\kappa}\sup_{\frac{\rho-r}{40^{N}}\leq R\leq\frac{\rho_{0}}{2}}\; \sup_{x\in B_{r}(x_{0})}\{\Psi_{M}(x,R)+\Upsilon(x,R)+\mathrm{Tail}(x,R)\}.
	\end{equation}
Now, we prove the lemma as below in five steps. 

\textbf{Step 1: Upper bound on $\lambda_{1}$.}
We estimate the upper bound of $\lambda_{1}$ as follows:
For any $x\in B_{r}(x_{0})$ and $\frac{\rho-r}{40^{N}}\leq R\leq\frac{\rho_{0}}{2}$, using the doubling property of $\mu$, we have 
\begin{equation*}
\begin{aligned}
 &\Upsilon(x,R)\leq c\left(\frac{2\rho_{0}}{\rho-r}\right)^{N+p}\Upsilon(x_{0},2\rho_{0})\quad\mbox{and }
 \Psi_{M}(x,R)\leq c\left(\frac{2\rho_{0}}{\rho-r}\right)^{N+p} \Psi_{M}(x_{0},2\rho_{0}).
			\end{aligned}
		\end{equation*}
 On the other hand, using H\"older's inequality and the similar tail estimates as in \eqref{eq31}, we see that
  \begin{equation*}
	\begin{aligned}
	 \mathrm{Tail}(x,R)&\leq c\sum_{i=0}^{l}\left(\frac{2\rho_{0}}{2^{i}R}\right)^{N+\epsilon p}\beta_{i}^{p-1}\left(\Xint-_{\mathcal{B}(x_{0},2\rho_{0})}G^{p'\alpha}d\mu\right)^{\frac{1}{p'\alpha}}+c\left(\Xint-_{\mathcal{B}(x_{0},2\rho_{0})}G^{p'\alpha}d\mu\right)^{\frac{1}{p'\alpha}}\\
	 &\quad+cT(u-(u)_{2\rho_{0},x_{0}};x,2\rho_{0})\\
	&\leq c \left(\frac{\rho_{0}}{\rho-r}\right)^{N+\epsilon p} \left(\Xint-_{\mathcal{B}(x_{0},2\rho_{0})}G^{p'\alpha}d\mu\right)^{\frac{1}{p'\alpha}}\\
	&\quad+c\left(\frac{\rho_{0}}{\rho-r}\right)^{N+sp}T(u-(u)_{2\rho_{0},x_{0}};x_{0},2\rho_{0}),
			\end{aligned}
		\end{equation*}
where in the last line, we have used the relation
	 \begin{equation*}
		 |y-x|\geq |y-x_{0}|-|x-x_{0}|\geq|y-x_{0}|\frac{\rho-r}{\rho_{0}},\quad \mbox{for }y\in B_{2\rho_{0}}(x)^c.
		\end{equation*}
  Thus, we get 
	\begin{equation}
	\label{upper bound of lambda1}
	 \begin{aligned}
		\lambda_{1}
		\leq \frac{c\left(\mathsf{data_{1}},M\right)}{\kappa}\left(\frac{\rho_{0}}{\rho-r}\right)^{N+p}\Xi_{0}.
			\end{aligned}
		\end{equation}

 \textbf{Step 2: Vitali Covering.} We start with an exit-time argument as in \cite{kuusi2} and \cite{SM} to cover the diagonal level set of $G$. We focus on handling the tail term which is different from the previous works. 
Define the diagonal level set of the functional $\Psi_{M}$ by
	\begin{equation}
		\label{diagona level set}
		D_{\kappa\lambda}:=\Bigg\{(x,x)\in \mathcal{B}(x_{0},r):\sup_{0\leq R\leq\frac{\rho-r}{40^{N}}}\Psi_{M}(x,R)>\kappa\lambda\Bigg\},
	\end{equation}
 for some $\lambda\geq\lambda_{1}$ which will be specified in  \eqref{nondiagonal level set estimate}. Note that for each $(x,x)\in\mathcal{B}(x_0,r)$ and $R\in [\frac{\rho-r}{40^N},\frac{\rho_0}{2}]$, we have $\Psi_M(x,R)\leq \kappa\lambda_1\leq \kappa\lambda$. Therefore, for each $(x,x)\in D_{\kappa\lambda}$, there exists a constant $0<R(x)\leq\frac{\rho-r}{40^{N}}$ such that  \begin{equation}
		\label{Psi exit}
		\Psi_{M}(x,R(x))\geq\kappa\lambda\quad\mbox{and}\quad\Psi_{M}(x,R)\leq\kappa\lambda,\quad\mbox{for any }R\in \Big(R(x),\frac{\rho-r}{40^N}\Big].
	\end{equation}
Using Vitali's covering lemma, we find that there is a collection $\{\mathcal{B}(x_{j},2R(x_{j}))\}_{j\in\mathbb{N}}$ of disjoint open sets with center $(x_{j},x_{j})\in D_{\kappa\lambda}$ such that
	\begin{equation}
	\label{off diagonal balls of super level set}
		D_{\kappa\lambda}\subset \bigcup_{j}\mathcal{B}(x_{j},10R(x_{j})).
	\end{equation}
 Let us write $R_{j}\equiv R(x_{j})$ and $\mathcal{B}_{j}\equiv \mathcal{B}(x_{j},R(x_{j}))$ for each positive integer $j$.
 From \eqref{Psi exit} and the doubling property of the measure $\mu$ (see Lemma \ref{measure mu}), we have
	\begin{equation*}
	\sum_{j}\int_{10\mathcal{B}_{j}}G^{p'}d\mu\leq \sum_{j}\mu(10\mathcal{B}_{j})[\Psi_{M}(x_{j},10R_j)]^{p'}\leq 10^{N+\epsilon p}(\kappa\lambda)^{p'}\sum_{j}\mu(\mathcal{B}_{j}).
	\end{equation*}
\textbf{Step 3: Off-diagonal estimate of $G$.} For this, we follow the method described in \cite[Subsection 4.3]{SM2}. Since we know that $u\in W^{s,p}(B_{2\rho_{0}}(x_{0}))\cap L^{\infty}(B_{2\rho_{0}}(x_{0}))$ with \eqref{off diagonal balls of super level set} and functions $G$ and $H$ which are described in \cite{SM2} are the same, an inspection of the subsection 4.3 in \cite{SM2} shows that it remains valid for our case, too.  Therefore, we have a desired result similar to \cite[Lemma 4.10]{SM2} as follows: There is a constant \begin{equation}
\label{condition of kappa}
    \kappa=\frac{\epsilon^{\frac{1}{p'\alpha}}}{c_{\kappa}}\quad \mbox{with } c_{\kappa}=c_{\kappa}(\mathsf{data_{1}})\geq1
\end{equation} such that
	\begin{equation}
	\label{nondiagonal level set estimate}
	  \int_{\mathcal{B}(x_{0},r)\cap\{G>\lambda\}}G^{p'}d\mu\leq 10^{N+p}\kappa^{p'}\lambda^{p'}\sum_{j}\mu(\mathcal{B}_{j})+c\lambda^{p'-p'\alpha}\int_{\mathcal{B}(x_{0},\rho)\cap\{G>\kappa\lambda\}}G^{p'\alpha}d\mu,
		\end{equation}
 for some constant $c_{0}=c_{0}(\mathsf{data_{1}})$, whenever 
	\begin{equation*}
	\lambda\geq \max\Bigg\{\lambda_{1},\frac{c_{1}}{\epsilon^{\frac{1}{p'}}}\left(\frac{\rho_{0}}{\rho-r}\right)^{2N+p}\Xi_{0}\Bigg\}=:\lambda_{2}.
   \end{equation*}

\textbf{Step 4: Estimate of $\mu(\mathcal{B}_{j})$.}
This step is to establish the existence of constants $c_{4}=c_{4}(\mathsf{data_{1}})$ and $c_{5}=c_{5}(\mathsf{data_{1}},\epsilon)$ such that
 \begin{equation}
  \label{key ball measure}
	\sum_{j}\mu(\mathcal{B}_{j})\leq\frac{c_{4}}{\epsilon^{p-p\alpha}\kappa^{p'\alpha}\lambda^{p'\alpha}}\int_{\mathcal{B}(x_{0},\rho)\cap\{G>\Tilde{\kappa}\kappa\lambda\}}G^{p'\alpha}d\mu+\frac{c_{5}\lambda_{1}^{\theta_{f}}}{
	(\hat{\kappa}\kappa\lambda)^{\Tilde{\theta}_{f}}}\int_{\mathcal{B}(x_{0},\rho)\cap\{F>\hat{\kappa}\kappa\lambda\}}F^{p_{*}+\mathfrak{A}}d\mu.
	\end{equation}
 Indeed, by \eqref{diagona level set}, it follows that at least one of the following inequalities hold:
	\begin{equation}
		\label{case1}
		\left(\Xint-_{\mathcal{B}_{j}}G^{p'}d\mu\right)^{\frac{1}{p'}}\geq\frac{\kappa\lambda}{2}\quad\mbox{or}
	\end{equation}
	\begin{equation}
		\label{case2}
		\frac{M[\mu(\mathcal{B}_{j})]^{\theta}}{\epsilon^{\frac{1}{p_{*}+\mathfrak{A}}-\frac{1}{p'}}}\left(\Xint-_{\mathcal{B}_{j}}F^{p_{*}+\mathfrak{A}}d\mu\right)^{\frac{1}{p_{*}+\mathfrak{A}}}\geq\frac{\kappa\lambda}{2}.    
	\end{equation}
\textbf{Case 1.} 
We assume that \eqref{case1} holds. Then, from \eqref{Reverse holder inequality}, we observe that
  \begin{equation}
	\label{case1 reverse holder}
	 \begin{aligned}
	 \kappa\lambda&\leq \frac{c\sigma^{-(p-1)}}{\epsilon^{\frac{1}{p'\alpha}-\frac{1}{p'}}}\left(\Xint-_{2\mathcal{B}_{j}}G^{p'\alpha}d\mu\right)^{\frac{1}{p'\alpha}}+\frac{c\sigma}{\epsilon^{\frac{1}{p'\alpha}-\frac{1}{p'}}}\sum_{i=0}^{l_{j}-1}\beta_{i}^{p-1}\left(\Xint-_{2^{i+1}\mathcal{B}_{j}}G^{p'\alpha}d\mu\right)^{\frac{1}{p'\alpha}}\\
	 &\quad +c\sigma\epsilon^{\frac{1}{p'}}[\epsilon\mu(2\mathcal{B}_{j})]^{\theta}T(u-(u)_{2^{l_{j}}R_{j},x_{0}};x_{j},2^{l_{j}}R_{j})
	 +\frac{c[\epsilon\mu(2\mathcal{B}_{j})]^{\theta}}{\epsilon^{\frac{1}{p_{*}+\mathfrak{A}}-\frac{1}{p'}}}\left(\Xint-_{2\mathcal{B}_{j}}F^{p_{*}+\mathfrak{A}}d\mu\right)^{\frac{1}{p_{*}+\mathfrak{A}}},
		\end{aligned}
	\end{equation}
 for some positive integer $l_{j}$ such that $\frac{\rho_{0}}{2}\leq 2^{l_{j}}R_{j}<\rho_{0}$. Note that since $R_{j}\leq\frac{\rho-r}{40^{N}}\leq\frac{\rho_{0}}{40^{N}}$, we have $l_{j}\geq3$. Therefore, by \eqref{Psi exit}, we see that
	\begin{equation}
		\label{case1 data}
	 \left(\Xint-_{2^{i}\mathcal{B}_{j}}G^{p'\alpha}d\mu\right)^{\frac{1}{p'\alpha}}\leq\kappa\lambda,
	\end{equation}
 for $i=0,1,\ldots,l_{j}-1$. With \eqref{defn of lambda1}, we have 
	\begin{equation}
	\label{case1 tail data}
	  \begin{aligned}
	  \kappa\lambda&\geq\mathrm{Tail}(x_{j},2^{l_{j}-1}R_{j})\\
	  &\geq\sum_{k=0}^{1}\beta_{k}^{p-1}\left(\Xint-_{2^{l_{j}-1+k}\mathcal{B}_{j}}G^{p'\alpha}d\mu\right)^{\frac{1}{p'\alpha}}+\epsilon^{\frac{1}{p'}}[\epsilon\mu(2\mathcal{B}_{j})]^{\theta}T(u-(u)_{2^{l_{j}}R_{j},x_{0}};x_{j},2^{l_{j}}R_{j})
		\end{aligned}
	\end{equation}
	and
	\begin{equation}
	  \label{case1 higher data}
	  \kappa\lambda\geq\frac{M[\mu(2\mathcal{B}_{j})]^{\theta}}{\epsilon^{\frac{1}{p_{*}+\mathfrak{A}}-\frac{1}{p'}}}\left(\Xint-_{2\mathcal{B}_{j}}F^{p_{*}+\mathfrak{A}}d\mu\right)^{\frac{1}{p_{*}+\mathfrak{A}}},
	\end{equation}
	where we have used the fact that
	\begin{equation*}
	    \frac{\rho-r}{40^{N}}\leq 2^{l_{j}-1}R_{j}<\frac{\rho_{0}}{2}.
	\end{equation*}
Applying \eqref{case1 data}, \eqref{case1 tail data} and \eqref{case1 higher data} to \eqref{case1 reverse holder}, we then discover that there are constant $c_{1}=c_{1}(\mathsf{data_{1}})$ and $c_{2}=c_{2}(\mathsf{data_{1}})$ such that
	\begin{align*}
		\kappa\lambda&\leq \frac{c\sigma^{-(p-1)}}{\epsilon^{\frac{1}{p'\alpha}-\frac{1}{p'}}}\left(\Xint-_{2\mathcal{B}_{j}}G^{p'\alpha}d\mu\right)^{\frac{1}{p'\alpha}}+\frac{c_{1}\sigma}{\epsilon^{\frac{1}{p'\alpha}-\frac{1}{p'}}}\kappa\lambda+\frac{c_{2}}{M}\kappa\lambda.
	\end{align*}
	By taking 
	\begin{equation}
		\label{condition of sigma and M}
		\sigma=\frac{\epsilon^{\frac{1}{p'\alpha}-\frac{1}{p'}}}{4c_{1}} \quad\mbox{and}\quad M=4c_{2},
	\end{equation}
	we see that
	\begin{equation*}
		\kappa\lambda\leq\frac{c}{\epsilon^{\frac{p}{p'\alpha}-\frac{p}{p'}}}\left(\Xint-_{2\mathcal{B}_{j}}G^{p'\alpha}d\mu\right)^{\frac{1}{p'\alpha}},
	\end{equation*}
	which yields
	\begin{equation*}
		\mu(\mathcal{B}_{j})(\kappa\lambda)^{p'\alpha}\leq\frac{c}{\epsilon^{p-p\alpha}}\int_{2\mathcal{B}_{j}}G^{p'\alpha}d\mu.
	\end{equation*}
 Since 
  \begin{equation*}
	\begin{aligned}
	  \frac{c}{\epsilon^{p-p\alpha}}\int_{2\mathcal{B}_{j}}G^{p'\alpha}d\mu&\leq\frac{c}{\epsilon^{p-p\alpha}}\left(\int_{2\mathcal{B}_{j}\cap\{G\leq\Tilde{\kappa}\kappa\lambda\}}G^{p'\alpha}d\mu+\int_{2\mathcal{B}_{j}\cap\{G\geq\Tilde{\kappa}\kappa\lambda\}}G^{p'\alpha}d\mu\right)\\
	  &\leq\frac{c_{3}}{\epsilon^{p-p\alpha}}\left(\left(\Tilde{\kappa}\kappa\lambda\right)^{p'\alpha}\mu(\mathcal{B}_{j})+\int_{2\mathcal{B}_{j}\cap\{G\geq\Tilde{\kappa}\kappa\lambda\}}G^{p'\alpha}d\mu\right),
		\end{aligned}
	\end{equation*}
 (thanks to the doubling property in Lemma \ref{measure mu}) by choosing $\Tilde{\kappa}=\frac{\epsilon^{\frac{p}{p'\alpha}-\frac{p}{p'}}}{(2c_{3})^{\frac{1}{p'\alpha}}}$, we obtain \begin{equation}
	\label{Case 1}
	 \mu(\mathcal{B}_{j})\leq\frac{c_{4}}{\epsilon^{p-p\alpha}(\kappa\lambda)^{p'\alpha}}\int_{2\mathcal{B}_{j}\cap\{G\geq\Tilde{\kappa}\kappa\lambda\}}G^{p'\alpha}d\mu,
  \end{equation}
 for some constant $c_{4}=c_{4}(\mathsf{data_{1}})$.
	
	\noindent
\textbf{Case 2.}
If \eqref{case2} occurs, we follow the proof exactly as in \cite[Subsection 4.2]{SM2} so that there is a constant $c_{5}=2\left(\frac{4M(L+1)}{\epsilon^{\frac{1}{p_{*}+\mathfrak{A}}-\frac{1}{p'}}}\right)^{\frac{p_{*}+\mathfrak{A}}{1-\left(p_{*}+\mathfrak{A}\right)\theta}}$ with $L=\mu(\mathcal{B}_{2})=c(N,p,\epsilon)$ such that
	\begin{align}\label{Case 2}
	 \mu(\mathcal{B}_{j})&\leq\frac{c_{5}\lambda_{1}^{(\left(p_{*}+\mathfrak{A}\right)+\delta_{f})\theta \left(p_{*}+\mathfrak{A}\right)/(1-\left(p_{*}+\mathfrak{A}\right)\theta)}}{(\hat{\kappa}\kappa\lambda)^{(1+\theta\delta_{f})\left(p_{*}+\mathfrak{A}\right)/(1-\left(p_{*}+\mathfrak{A}\right)\theta)}}\int_{\mathcal{B}_{j}\cap\{F>\hat{\kappa}\kappa\lambda\}}F^{p_{*}+\mathfrak{A}}d\mu \nonumber\\
	 &=\frac{c_{5}\lambda_{1}^{\theta_{f}}}{\left(\hat{\kappa}\kappa\lambda\right)^{\Tilde{\theta}_{f}}}\int_{\mathcal{B}_{j}\cap\{F>\hat{\kappa}\kappa\lambda\}}F^{p_{*}+\mathfrak{A}}d\mu,
	\end{align}
 provided 
 \begin{equation}
 \label{hatkappa}
     \hat{\kappa}\leq\left(\frac{1}{4}\right)^{\frac{1-\left(p_{*}+\mathfrak{A}\right)\theta}{p_{*}+\mathfrak{A}}}\frac{\epsilon^{\frac{1}{p_{*}+\mathfrak{A}}-\frac{1}{p'}}}{4M(L+1)}.
 \end{equation}
 Since $\left\{2\mathcal{B}_{j}\right\}$ is a collection of disjoint open sets contained in $\mathcal{B}(x_{0},\rho)$,  the two estimates in \eqref{Case 1} and \eqref{Case 2} imply \eqref{key ball measure}.

\textbf{Step 5: Conclusion.}
We are now ready to complete the proof. An elementary calculation gives
 \begin{equation*}
   \underbrace{\int_{\mathcal{B}(x_{0},r)\cap\left\{G>\Tilde{\kappa}\kappa\lambda\right\}}G^{p'}d\mu}_{I_{1}}\leq \underbrace{\lambda^{p'-p'\alpha}\int_{\mathcal{B}(x_{0},r)\cap\left\{G>\Tilde{\kappa}\kappa\lambda\right\}}G^{p'\alpha}d\mu}_{I_{2}}+\underbrace{\int_{\mathcal{B}(x_{0},r)\cap\left\{G>\lambda\right\}}G^{p'}d\mu}_{I_{3}}.
 \end{equation*}
 By \eqref{nondiagonal level set estimate}, we have
 \begin{align*}
     I_{3}\leq \underbrace{10^{N+p}\kappa^{p'}\lambda^{p'}\sum_{j}\mu(\mathcal{B}_{j})}_{I_{3,1}}+\underbrace{c\lambda^{p'-p'\alpha}\int_{\mathcal{B}(x_{0},\rho)\cap\{G>\kappa\lambda\}}G^{p'\alpha}d\mu}_{I_{3,2}},
 \end{align*}
 for any $\lambda\geq\lambda_{2}.$
 Using \eqref{key ball measure}, we have 
 \begin{align*}
     I_{3,1}\leq\frac{10^{N+p}c_{4}(\kappa\lambda)^{p'}}{\epsilon^{p-p\alpha}\kappa^{p'\alpha}\lambda^{p'\alpha}}\int_{\mathcal{B}(x_{0},\rho)\cap\{G>\Tilde{\kappa}\kappa\lambda\}}G^{p'\alpha}d\mu+\frac{10^{N+p}c_{5}(\kappa\lambda)^{p'}\lambda_{1}^{\theta_{f}}}{
	(\hat{\kappa}\kappa\lambda)^{\Tilde{\theta}_{f}}}\int_{\mathcal{B}(x_{0},\rho)\cap\{F>\hat{\kappa}\kappa\lambda\}}F^{p_{*}+\mathfrak{A}}d\mu,
 \end{align*}
 where $c_{4}=c_{4}(\mathsf{data_{1}})$ and $c_{5}=c_{5}(\mathsf{data_{1}},\epsilon)$. Note that \begin{equation}
 \label{kappa epsilon}
    \frac{\kappa^{p'-p'\alpha}}{\epsilon^{p-p\alpha}}=c\frac{\epsilon^{\frac{1-\alpha}{\alpha}}}{\epsilon^{p(1-\alpha)}}=c\epsilon^{-p+\frac{p}{m}}\geq c,
\end{equation} where we have used \eqref{condition of kappa}, \eqref{palphabeta} and $\epsilon\in(0,1)$. Therefore, combining the above estimate with $I_{3}$ and using \eqref{kappa epsilon}, we obtain
\begin{align*}
    I_{1}\leq \frac{c(\kappa\lambda)^{p'}}{\epsilon^{p-p\alpha}\kappa^{p'\alpha}\lambda^{p'\alpha}}\int_{\mathcal{B}(x_{0},\rho)\cap\{G>\Tilde{\kappa}\kappa\lambda\}}G^{p'\alpha}d\mu+\frac{10^{N+p}c_{5}(\kappa\lambda)^{p'}\lambda_{1}^{\theta_{f}}}{
	(\hat{\kappa}\kappa\lambda)^{\Tilde{\theta}_{f}}}\int_{\mathcal{B}(x_{0},\rho)\cap\{F>\hat{\kappa}\kappa\lambda\}}F^{p_{*}+\mathfrak{A}}d\mu.
\end{align*}
After some elementary algebraic manipulations, we observe that  
\begin{equation}
\label{almost final}
\begin{aligned}
    I_{1}\leq \frac{c(\Tilde{\kappa}\kappa\lambda)^{p'-p'\alpha}}{\epsilon^{p-p\alpha}\left(\Tilde{\kappa}\kappa\right)^{p'-p'\alpha}}\int_{\mathcal{B}(x_{0},\rho)\cap\{G>\Tilde{\kappa}\kappa\lambda\}}G^{p'\alpha}d\mu+\frac{10^{N+p}c_{5}(\kappa\lambda)^{p'}\lambda_{1}^{\theta_{f}}}{
	(\hat{\kappa}\kappa\lambda)^{\Tilde{\theta}_{f}}}\int_{\mathcal{B}(x_{0},\rho)\cap\{F>\hat{\kappa}\kappa\lambda\}}F^{p_{*}+\mathfrak{A}}d\mu,
\end{aligned}
\end{equation}
whenever $\lambda\geq\lambda_{2}$. We reformulate estimate \eqref{almost final} as follows:
\begin{equation*}
\begin{aligned}
 \int_{\mathcal{B}(x_{0},r)\cap\left\{G>\lambda\right\}}G^{p'}d\mu&\leq \frac{c}{\epsilon^{p-p\alpha}\left(\Tilde{\kappa}\kappa\right)^{p'-p'\alpha}}\lambda^{p'-p'\alpha}\int_{\mathcal{B}(x_{0},\rho)\cap\{G>\lambda\}}G^{p'\alpha}d\mu\\
 &\quad+\frac{c_{6}(\mathsf{data_{1}},\epsilon)\lambda_{1}^{\theta_{f}}}{\lambda^{\Tilde{\theta}_{f}-p'}}\int_{\mathcal{B}(x_{0},\rho)\cap\{F>\frac{\hat{\kappa}}{\Tilde{\kappa}}\lambda\}}F^{p_{*}+\mathfrak{A}}d\mu,
\end{aligned}
\end{equation*}
provided $\lambda\geq\Tilde{\kappa}\kappa\lambda_{2}$. Now we take a number $\hat{\kappa}>0$ sufficiently small so that \eqref{hatkappa} and $\kappa_{f}:=\frac{\hat{\kappa}}{\Tilde{\kappa}}\leq 1$ hold. Since $\Tilde{\kappa}\kappa=\frac{\epsilon^{\frac{1+p}{p'\alpha}-\frac{p}{p'}}}{c}$ for some constant $c=c(\mathsf{data_{1}})>1$ and $\lambda_{0}\geq \Tilde{\kappa}\kappa\lambda_{2}$ by \eqref{upper bound of lambda1}, we conclude that \eqref{level set estimate} holds whenever $\lambda\geq\lambda_{0}$. 
\qed
\end{proof}

\begin{Lemma}
\label{key result of SI}
Let $u$ be a local weak solution to \eqref{probM}. Take $\mathcal{B}(x_{0},2\rho_{0})\subset\widetilde{\Omega}\times\widetilde{\Omega}$ with $0<\rho_{0}\leq1$ and $x_{0}\in \widehat{\Omega}$ and write $\mathcal{B}\equiv\mathcal{B}(x_{0},\rho_{0})$ Then there exist positive constants $\epsilon$, $\delta\in(0,1)$, $\delta_{f}\in\left(0,\frac{\delta_{0}}{2}\right)$ and $c$ depending on $\mathsf{data_{1}}$ and $\delta_{0}$ such that 
	\begin{equation*}
	\begin{aligned}
		\left(\Xint-_{\frac{1}{2}\mathcal{B}}G(x,y,U)^{p'+\delta}d\mu\right)^{\frac{1}{p'+\delta}}&\leq c\left(\Xint-_{2\mathcal{B}}G(x,y,U)^{p'}d\mu\right)^{\frac{1}{p'}}+cT(u-(u)_{2\rho_{0},x_{0}};x_{0},2\rho_{0})\\
	 &\quad+c\left(\Xint-_{2\mathcal{B}}F^{p_{*}+\mathfrak{A}+\delta_{f}}d\mu\right)^{\frac{1}{p_{*}+\mathfrak{A}+\delta_{f}}}.
		\end{aligned}
		\end{equation*}
\end{Lemma}
\begin{proof}
Let $\frac{\rho_{0}}{2}<r<\rho<\rho_{0}$.
We now set the parameters $\delta$, $\delta_{f}$, $\Tilde{\delta}$ and $\epsilon$ depending only on $\mathsf{data_{1}}$ and $\delta_{0}$ such that
\begin{equation}
\label{s parameter condition}
  \frac{c_{\alpha}\delta}{\epsilon^{\theta_{u}}(p'-p'\alpha+\delta)}\leq\frac{1}{16}\quad\mbox{and}\quad \delta<p'-\Tilde{\theta}_{f}+\delta+\Tilde{\delta}<\delta_{f}.
\end{equation}
To this end, we consider the two cases depending on the relationship between $N$ and $sp$.

\noindent
\textbf{Case 1.} If $sp\geq N$.
 Set $\delta_{f}=\min\left\{\frac{\delta_{0}}{2},\frac{1}{2(p-1)}\right\}$
and choose
\begin{equation}
\label{epsilon spN}
    \epsilon<\min\left\{\frac{s}{p},(1-s)\right\}
\end{equation}
such that
\begin{equation}
\label{condition of tdelte}
    \frac{1}{p}-\frac{p-1}{2p}\delta_{f}<\theta,
\end{equation}
which is possible because 
\begin{equation*}
    \theta=\frac{s-\epsilon(p-1)}{N+\epsilon p}<\frac{s}{N}
\end{equation*}
is a decreasing function with respect to $\epsilon$ and $sp\geq N$.
Since 
\begin{equation*}
\begin{aligned}
\frac{1}{p-1}-\mathfrak{A}-\frac{\delta_{f}}{2}\left(1+\left(1+\mathfrak{A}\right)\theta\right)<\frac{1}{p-1}-\frac{\delta_{f}}{2}\quad\text{and}\quad p'\theta<p'\left(1+\mathfrak{A}\right)\theta,
\end{aligned}
\end{equation*}
 by \eqref{condition of tdelte}, we see that
\begin{equation*}
    \frac{1}{p-1}-\mathfrak{A}-\frac{\delta_{f}}{2}\left(1+\left(1+\mathfrak{A}\right)\theta\right)<p'\left(1+\mathfrak{A}\right)\theta,
\end{equation*}
which can be rewritten as
\begin{equation*}
    p'\left(1-\left(1+\mathfrak{A}\right)\theta\right)-(1+\mathfrak{A})<\frac{\delta_{f}}{2}\left(1+\left(1+\mathfrak{A}\right)\theta\right),
\end{equation*}
where we have used the following fact:
\begin{equation*}
    \frac{1}{p-1}-\mathfrak{A}=p'-(1+\mathfrak{A}).
\end{equation*}
Dividing each side by $(1-(1+\mathfrak{A})\theta)$, we find that
\begin{equation*}
    p'-\frac{(1+\mathfrak{A})}{(1-(1+\mathfrak{A})\theta}<\frac{\delta_{f}}{2}\frac{\left(1+\left(1+\mathfrak{A}\right)\theta\right)}{(1-(1+\mathfrak{A})\theta)}=\frac{\delta_{f}}{2}+\frac{\left(\left(1+\mathfrak{A}\right)\theta\right)}{(1-(1+\mathfrak{A})\theta)}\delta_{f},
\end{equation*}
which is equivalent to 
\begin{equation}
\label{p'andtdelta}
p'-\Tilde{\theta}_{f}<\frac{\delta_{f}}{2}.
\end{equation}
Let $\delta$ be any non-negative number satisfying
\begin{equation}
\label{condition of delta}
\delta<\min\left\{\frac{spp'\epsilon^{(sp^{2}+sp+N)/N}}{16c_{\alpha}(N+sp)},\frac{1}{8}\delta_{f}\right\}.
\end{equation}
In light of \eqref{epsilon spN} and \eqref{condition of delta}, we see that
\begin{equation*}
    \frac{spp'\epsilon^{(sp^{2}+sp+N)/N}}{16c_{\alpha}(N+sp+p)}\leq \frac{s^{2}p}{(p-1)16c_{\alpha}(N+sp+p)}\leq \frac{1}{(p-1)16c_{\alpha}}
\end{equation*}
and
\begin{equation*}
    2\delta<\delta_{f}\quad\text{with}\quad
    \frac{c_{\alpha}\delta}{\epsilon^{\theta_{u}}(p'-p'\alpha+\delta)}\leq\frac{1}{16}.
\end{equation*}
Moreover, we note by \eqref{condition of tdelte} and \eqref{condition of A} that
\begin{equation}
\label{p'andtdelta2}
    p'-\Tilde{\theta}_{f}+\theta_{f}=p'-(p_{*}+\mathfrak{A})\geq\frac{1}{2(p-1)}> \delta.
\end{equation}
Combine \eqref{p'andtdelta} and \eqref{p'andtdelta2} to find a constant $\Tilde{\delta}\in[0,\theta_{f}]$ such that
\begin{equation*}
    \delta<p'-\Tilde{\theta}_{f}+\delta+\Tilde{\delta}<\delta_{f}.
\end{equation*}
%%%%
\noindent
\textbf{Case 2.} If $sp<N$.
With an elementary algebraic manipulation as in \cite[Theorem 5.1]{SM}, we find that there are constants $\epsilon\in(0,1)$ and $\delta_{f}\in\left(0,\min\left\{\frac{1}{2(p-1)},\frac{\delta_{0}}{2}\right\}\right)$  depending on $\mathsf{data_{1}}$ and $\delta_{0}$ satisfying the following:
\begin{equation}
\label{condition of last}
\begin{aligned}
 \frac{\epsilon p(p')^{2}}{N+\epsilon p}<\delta_{f}\leq\frac{\epsilon p(N+sp')}{N(s-\epsilon(p-1))}.
\end{aligned}
\end{equation}
Let $\delta$ be any nonnegative number such that
\begin{equation}
\label{condition of delta2}
\delta\leq\min\left\{\frac{spp'\epsilon^{(sp^{2}+sp)/N}}{16c_{\alpha}(N+sp+p)},\frac{1}{p-1}\frac{\epsilon pp'(N+sp')}{N^{2}+2N\epsilon p+\epsilon spp'}\right\},
\end{equation}
where $c_{\alpha}=c_{\alpha}(\mathsf{data_{1}})$ is determined as in \eqref{level set estimate2}.
By proceeding exactly as in \cite[Theorem 5.1]{SM}, we find that the conditions \eqref{condition of last} and \eqref{condition of delta2} imply \eqref{s parameter condition} by taking $\Tilde{\delta}=0$.

\noindent
From the choice of $\delta$, $\delta_{f}$, $\Tilde{\delta}$ and $\epsilon$ satisfying \eqref{s parameter condition}, we now prove the higher integrability of $G$.
We first apply Lemma \ref{level set estimate lemma} with $\epsilon$, $\delta_{f}$ and  $\delta$ which satisfy \eqref{s parameter condition}. Then we see that there are some constants $c_{\alpha}=c_{\alpha}(\mathsf{data_{1}})\geq1$,  $c_{f}=c_{f}(\mathsf{data_{1}},\delta_{0})\geq1$ and $\kappa_{f}=\kappa_{f}(\mathsf{data_{1}},\delta_{0})\in(0,1)$ such that
	\begin{equation}\label{level set estimate2}
	 \frac{1}{\lambda^{p'}}\int_{\mathcal{B}_{r}\cap\{G>\lambda\}}G^{p'}d\mu\leq \frac{c_{\alpha}}{\epsilon^{\theta_{u}}\lambda^{p'\alpha}}\int_{\mathcal{B}_{\rho}\cap\{G>\lambda\}}G^{p'\alpha}d\mu+\frac{c_{f}\lambda_{0}^{\theta_{f}}}{\lambda^{\Tilde{\theta_{f}}}}\int_{\mathcal{B}_{\rho}\cap\{F>\kappa_{f}\lambda\}}F^{p_{*}+\mathfrak{A}}d\mu,
		\end{equation}
  whenever $\lambda\geq\lambda_{0}$ with
	\begin{equation*}
	 \lambda_{0}:=c_{0}(\mathsf{data_{1}},\delta_{0})\left(\frac{\rho_{0}}{\rho-r}\right)^{2N+p}\Xi_{0}.
		\end{equation*} 		
Let us define a truncated function $G_{m}(x,y)=\min\{G(x,y),m\}$ for $(x,y)\in\mathcal{B}_{2\rho_{0}}$ with $m>\lambda_{0}$ and a measure $d\nu=G^{p'}d\mu$ in $\mathcal{B}_{2\rho_{0}}$.
We then observe that 
\begin{equation*}
\begin{aligned}
 \int_{\mathcal{B}_{r}}G_{m}^{\delta}G^{p'}d\mu&=\int_{\mathcal{B}_{r}}G_{m}^{\delta}d\nu\\
 &=\delta\int_{0}^{\infty}\lambda^{\delta-1}\nu\left(\mathcal{B}_{r}\cap\left\{G_{m}>\lambda\right\}\right)d\lambda\\
 &=\delta\int_{0}^{\lambda_{0}}\lambda^{\delta-1}\nu\left(\mathcal{B}_{r}\cap\left\{G_{m}>\lambda\right\}\right)d\lambda+\delta\int_{\lambda_{0}}^{\infty}\lambda^{\delta-1}\nu\left(\mathcal{B}_{r}\cap\left\{G_{m}>\lambda\right\}\right)d\lambda\\
 &\leq \lambda_{0}^{\delta}\int_{\mathcal{B}_{r}}G^{p'}d\mu+\delta\int_{\lambda_{0}}^{\infty}\lambda^{\delta-1}\nu\left(\mathcal{B}_{r}\cap\left\{G_{m}>\lambda\right\}\right)d\lambda\\
 &=\underbrace{\lambda_{0}^{\delta}\int_{\mathcal{B}_{r}}G^{p'}d\mu}_{I_{1}}+\underbrace{\delta\int_{\lambda_{0}}^{m}\lambda^{\delta-1}\int_{\mathcal{B}_{r}\cap\{G>\lambda\}}G^{p'}d\mu d\lambda}_{I_{2}},
\end{aligned}
\end{equation*}
where we have used an integral formula of a distribution function of $G$. We next estimate $I_{1}$ and $I_{2}$ as follows:\\
\textbf{Estimate of $I_{1}$.} By the definition of $\lambda_{0}$, we find 
\begin{equation*}
    I_{1}\leq \lambda_{0}^{\delta}\mu(\mathcal{B}_{2\rho_{0}})\Xint-_{\mathcal{B}_{2\rho_{0}}}G^{p'}d\mu\leq 
    c\lambda_{0}^{p'+\delta}\mu(\mathcal{B}_{2\rho_{0}}).
\end{equation*}
%%%
\textbf{Estimate of $I_{2}$.} Using \eqref{level set estimate2}, we discover that
\begin{equation*}
\begin{aligned}
I_{2}\leq \underbrace{\delta\int_{\lambda_{0}}^{m}\lambda^{\delta-1}\frac{c_{\alpha}\lambda^{p'}}{\epsilon^{\theta_{u}}\lambda^{p'\alpha}}\int_{\mathcal{B}_{\rho}\cap\{G>\lambda\}}G^{p'\alpha}d\mu d\lambda}_{I_{2,1}}+\underbrace{\delta\int_{\lambda_{0}}^{m}\lambda^{\delta-1}\frac{c_{f}\lambda_{0}^{\theta_{f}}\lambda^{p'}}{\lambda^{\Tilde{\theta_{f}}}}\int_{\mathcal{B}_{\rho}\cap\{F>\kappa_{f}\lambda\}}F^{p_{*}+\mathfrak{A}}d\mu d\lambda}_{I_{2,2}}.
\end{aligned}
\end{equation*}
By \eqref{s parameter condition}, we see that 
\begin{equation*}
\begin{aligned}
I_{2,1}&\leq \frac{c_{\alpha}\delta}{\epsilon^{\theta_{u}}}\int_{0}^{\infty}\lambda^{p'-p'\alpha+\delta-1}\int_{\mathcal{B}_{\rho}\cap\left\{G_{m}>\lambda\right\}}G^{p'\alpha}d\mu d\lambda\\
&=\frac{c_{\alpha}\delta}{\epsilon^{\theta_{u}}(p'-p'\alpha+\delta)}\int_{\mathcal{B}_{\rho}}G_{m}^{\delta+p'-p'\alpha}G^{p'\alpha}d\mu\\
&\leq \frac{1}{16}\int_{\mathcal{B}_{\rho}}G_{m}^{\delta+p'-p'\alpha}G^{p'\alpha}d\mu\leq\frac{1}{16}\int_{\mathcal{B}_{\rho}}G_{m}^{\delta}G^{p'}d\mu.
\end{aligned}
\end{equation*}
We next estimate $I_{2,2}$ as follows:
\begin{equation*}
\begin{aligned}
I_{2,2}&\leq \lambda_{0}^{\theta_{f}-\Tilde{\delta}} \delta\int_{\lambda_{0}}^{m}c_{f}\lambda^{p'-\Tilde{\theta}_{f}+\delta+\Tilde{\delta}-1}\int_{\mathcal{B}_{\rho}\cap\{F>\kappa_{f}\lambda\}}F^{p_{*}+\mathfrak{A}}d\mu d\lambda\\
&\leq c\lambda_{0}^{\theta_{f}-\Tilde{\delta}}\mu(\mathcal{B}_{2\rho_{0}})\Xint-_{\mathcal{B}_{2\rho_{0}}}F^{p_{*}+\mathfrak{A}+\delta+p'-\Tilde{\theta}_{f}+\Tilde{\delta}}d\mu\\
&\leq c\mu(\mathcal{B}_{2\rho_{0}})\lambda_{0}^{\theta_{f}-\Tilde{\delta}}\left(\Xint-_{\mathcal{B}_{2\rho_{0}}}F^{p_{*}+\mathfrak{A}+\delta_{f}}d\mu\right)^\frac{p_{*}+\mathfrak{A}+\delta+\Tilde{\delta}+p'-\Tilde{\theta}_{f}}{p_{*}+\mathfrak{A}+\delta_{f}}\\
&\leq c\mu(\mathcal{B}_{2\rho_{0}})\lambda_{0}^{\theta_{f}-\Tilde{\delta}}\Upsilon_{0}(x_{0},2\rho_{0})^{p_{*}+\mathfrak{A}+\delta+\Tilde{\delta}+p'-\Tilde{\theta}_{f}}\leq c(\mathsf{data_{1}},\delta_{0})\mu(\mathcal{B}_{2\rho_{0}})\lambda_{0}^{p'+\delta},
\end{aligned}
\end{equation*}
where we have used \eqref{s parameter condition} with H\"older's inequality, \eqref{theta0} and \eqref{def of lambda0}. We combine estimates $I_{1}$ and $I_{2}$ to obtain
\begin{equation*}
    \int_{\mathcal{B}_{r}}G_{m}^{\delta}G^{p'}d\mu\leq \frac{1}{16}\int_{\mathcal{B}_{\rho}}G_{m}^{\delta}G^{p'}d\mu+c\mu(\mathcal{B}_{2\rho_{0}})\lambda_{0}^{p'+\delta}.
\end{equation*}
Due to the doubling property and \eqref{theta0}, we discover that
\begin{equation*}
\begin{aligned}
    \left(\frac{\mu(\mathcal{B}_{r})}{\mu(\mathcal{B}_{\rho})}\Xint-_{\mathcal{B}_{r}}G_{m}^{\delta}G^{p'}d\mu\right)^{\frac{1}{p'+\delta}}\leq \left(\frac{1}{16}\Xint-_{\mathcal{B}_{\rho}}G_{m}^{\delta}G^{p'}d\mu\right)^{\frac{1}{p'+\delta}}+c\left(\frac{\mu(\mathcal{B}_{2\rho_{0}})}{\mu(\mathcal{B}_{\rho})}\right)^{\frac{1}{p'+\delta}}\lambda_{0},
\end{aligned}    
\end{equation*}
and a few elementary manipulations with \eqref{def of lambda0} gives
\begin{equation*}
    \left(\Xint-_{\mathcal{B}_{r}}G_{m}^{\delta}G^{p'}d\mu\right)^{\frac{1}{p'+\delta}}\leq \frac{1}{2}\left(\Xint-_{\mathcal{B}_{\rho}}G_{m}^{\delta}G^{p'}d\mu\right)^{\frac{1}{p'+\delta}}+c\left(\frac{\rho_{0}}{\rho-r}\right)^{2N+p}\Xi_{0}.
\end{equation*}
Therefore, we rewrite the above inequality as
\begin{equation*}
    \varphi(r)\leq\frac{1}{2}\varphi(\rho)+c\left(\frac{\rho_{0}}{\rho-r}\right)^{2N+p}\Xi_{0},
\end{equation*}
where we have defined $\varphi(\tau):=\left(\Xint-_{\mathcal{B}_{\tau}}G_{m}^{\delta}G^{p'}d\mu\right)^{\frac{1}{p'+\delta}}$ for $\tau\in[\frac{\rho_{0}}{2},\rho_{0}]$. By Lemma \ref{technial lemma}, we obtain \begin{equation*}
    \Bigg(\Xint-_{\mathcal{B}_{\frac{\rho_{0}}{2}}}G_{m}^{\delta}G^{p'}d\mu\Bigg)^{\frac{1}{p'+\delta}}\leq c\Xi_{0},
\end{equation*}
where $c=c(\mathsf{data_{1}},\delta_{0})$ is independent of $m$. Thus, by taking $m\to\infty$, we conclude that
\begin{equation*}
    \Bigg(\Xint-_{\mathcal{B}_{\frac{\rho_{0}}{2}}}G^{p'+\delta}d\mu\Bigg)^{\frac{1}{p'+\delta}}\leq c\Xi_{0}.
\end{equation*}
Recalling the definition of $\Xi_0$ as stated in \eqref{theta0}, we complete the proof of the lemma.
\qed
\end{proof}
Since we have obtained a higher integrability result for $G$, we now prove our first main result.

\textbf{Proof of Theorem \ref{Main theorem1}.} For any $x_{0}\in \widehat{\Omega}$, there is a $\rho_{0}\in(0,1]$ such that
\begin{equation*}
    B_{2\rho_{0}}(x_{0})\subset\widetilde{\Omega}.
\end{equation*}
We now need to normalize the solution $u$. Define
\begin{equation*}
\begin{aligned}
 \Tilde{u}(x)=u(\rho_{0}x+x_{0}),\quad \Tilde{f}(x)=\rho_{0}^{sp}f(\rho_{0}x+x_{0})\quad \mbox{for }x\in B_{2};
\end{aligned}
\end{equation*}
\begin{equation*}
\begin{aligned}
 &\Tilde{a}(x,y,z,w)=a(\rho_{0}x+x_{0},\rho_{0}y+x_{0},z,w),\\ 
 &\Tilde{b}(x,y)=\rho_{0}^{sp-tq}b(\rho_{0}x+x_{0},\rho_{0}y+x_{0})\quad\mbox{for } x,y\in \mathbb{R}^{N}\times\mathbb{R}^{N}.
\end{aligned}
\end{equation*}
Then we have 
\begin{equation*}
    \mathcal{L}_{\Tilde{a},\Tilde{b}}\Tilde{u}=\Tilde{f}\quad \mbox{in }B_{2},
\end{equation*}with 
\begin{align*}
			&0<\Lambda^{-1} \leq \Tilde{a}(x,y,z,w)\leq \Lambda  \quad\mbox{and } \\
			&0 \leq \Tilde{b}(x,y)\leq \rho_{0}^{sp-tq}\Lambda.
\end{align*}
For $\mathcal{B}:=B_1\times B_1$, by Lemma \ref{key result of SI}, there are sufficiently small positive numbers $\delta_{1}$, $\epsilon\in(0,1)$, $\delta_{f}\in(0,\delta_{0})$ and $c$ depending on $\mathsf{data_{1}}$ and $\delta_{0}$ such that 
	\begin{align}\label{G estimate}
		\left(\Xint-_{\frac{1}{2}\mathcal{B}}G(x,y,\Tilde{U})^{p'(1+\delta_{1})}d\mu\right)^{\frac{1}{p'(1+\delta_{1})}}&\leq c\left(\Xint-_{2\mathcal{B}}G(x,y,\Tilde{U})^{p'}d\mu\right)^{\frac{1}{p'}}+c\Tilde{T}(\Tilde{u}-(\Tilde{u})_{2,0};0,2\rho_{0})\nonumber\\
	 &\quad+c\left(\Xint-_{2\mathcal{B}}\Tilde{F}^{p_{*}+\mathfrak{A}+\delta_{f}}d\mu\right)^{\frac{1}{p_{*}+\mathfrak{A}+\delta_{f}}},
	\end{align}
and 
\begin{equation}
\label{range of delta1}
s+\frac{p\epsilon\delta_{1}}{p(1+\delta_{1})}<1,
\end{equation}
where 
\begin{equation*}
    \Tilde{T}(\Tilde{u}-(\Tilde{u})_{2,0};0,2\rho_{0}):=\int_{\mathbb{R}^{N}\setminus B_{2}}\left(\frac{|\Tilde{u}(y)-(\Tilde{u})_{2,0}|^{p-1}}{|y|^{N+sp}}+\|\Tilde{b}\|_{L^{\infty}}\frac{|\Tilde{u}(y)-(\Tilde{u})_{2,0}|^{q-1}}{|y|^{N+tq}}\right)dy.
\end{equation*} Since $\Tilde{u}\in L^{\infty}(B_{2})$ and $tq\leq sp$, we see that 
\begin{equation}
\label{bounded of u and G}
    \int_{2\mathcal{B}}G(x,y,\Tilde{U})^{p'}d\mu\leq c[\Tilde{u}]_{W^{s,p}(B_{2})}^{p}.
\end{equation}
From \cite[Proposition 2.5]{Ni} with \eqref{range of delta1}, we discover that
\begin{equation}
\label{embedding for last}
    [\Tilde{u}]^{p-1}_{W^{s+\frac{N\delta}{p(1+\delta)},p(1+\delta)}(B_{1/2})}\leq c[\Tilde{u}]^{p-1}_{W^{s+\frac{p\epsilon\delta_{1}}{p(1+\delta_{1})},p(1+\delta_{1})}(B_{1/2})}\leq c\left(\int_{\frac{1}{2}\mathcal{B}}G(x,y,\Tilde{U})^{p'(1+\delta_{1})}d\mu\right)^{\frac{1}{p'(1+\delta_{1})}},
\end{equation}
where $\delta=\delta(\mathsf{data_{1}},\delta_{0})$ is a sufficiently small positive number such that
	\begin{equation*}
    \frac{N\delta}{p(1+\delta)}< \frac{p\epsilon\delta_{1}}{p(1+\delta_{1})}\quad\mbox{and}\quad p\delta\leq p\delta_{1}.
\end{equation*}
We combine the estimates \eqref{G estimate}, \eqref{bounded of u and G} and \eqref{embedding for last} to obtain
\begin{equation*}
\begin{aligned}
    &\left(\int_{B_{\frac{1}{2}}}\int_{B_{\frac{1}{2}}}\left(\frac{|\Tilde{u}(x)-\Tilde{u}(y)|^{p}}{|x-y|^{N+ps}}\right)^{(1+\delta)}dxdy\right)^{\frac{1}{p'(1+\delta)}}\\
    &\quad\leq c\Bigg[[\Tilde{u}]_{W^{s,p}(B_{2})}^{p-1}+\Tilde{T}(\Tilde{u}-(\Tilde{u})_{2,0};0,2)+\left(\int_{B_{2}}|\Tilde{f}(x)|^{p_{*}+\mathfrak{A}+\delta_{f}}dx\right)^{\frac{1}{p_{*}+\mathfrak{A}+\delta_{f}}}\Bigg].
\end{aligned}
\end{equation*}
By scaling back, noting $u\in L^{\infty}(B_{2\rho_{0}}(x_{0}))$ and using H\"older's inequality with $\mathfrak{A}+\delta_{f}<\delta_{0}$, we conclude  the estimate \eqref{self improving estimate}. Finally, the standard covering argument gives that $u\in W^{s+\frac{N\delta}{p(1+\delta)},p(1+\delta)}_{\mathrm{loc}}(\widetilde\Omega)$. 
\qed

\section{The H\"older continuity}
We first focus on a local weak solution 
\begin{equation*}
    u\in \mathcal{W}_{\mathrm{loc}}(\Omega)\cap L^{p-1}_{ps}(\mathbb{R}^N)\cap L^{q-1}_{qt}(\mathbb{R}^N)
\end{equation*}
to
\begin{equation*}
    \mathcal{L}_{a,b}\ u=f\quad\text{in }\Omega,
    \tag{$\mathcal{PA}$}\label{probA}
\end{equation*}
where the coefficient function $a$ is an $VMO$ function and is independent of the solution $u$, and $f\in L^\gamma_{\mathrm{loc}}(\Omega)$ with $\gamma>\max\{1,\frac{N}{ps}\}$.
Then, from \cite[Theorem 4.5]{JDSPDPP} and using Caccioppoli type estimate of Lemma \ref{Caccio estimate} (to control the quantity $[u]_{W^{s,p}}$, appearing there), we can get the following continuity result.
\begin{Lemma}\label{thmingenk}
Suppose that $2\leq p\leq q\leq ps/t$ and that the functions $a(\cdot,\cdot)$ and $b(\cdot,\cdot)$ are locally translation invariant in $\Omega\times\Omega$. Let $u$ be a local weak solution to the problem \eqref{probA} with $f\equiv 0$. Then $u\in C^{\alpha}_\mathrm{ loc}(\Omega)$ for all $\alpha\in (0,{\Theta_{0}})$, where $\Theta_{0}:=\min\left\{\frac{ps}{p-1},1\right\}$.\\
More precisely, for $B_{2\rho_0}\equiv B_{2\rho_0}(x_{0})\Subset\Omega$ with $\rho_0\in (0,1]$ and for all $\alpha\in (0,{\Theta_0})$, there exists a positive constant $c$ depending only on $\mathsf{data}$ and $\alpha$ such that 
	\begin{align*}%\label{eqintbd}
	 [u]_{C^{\alpha}(B_{\rho_0/4})}
	  \leq \frac{c}{\rho_0^{\alpha}}\Big[\|u\|_{L^\infty(B_{\rho_0/2})}+1+T_{ps}(u;x_0,\rho_0/2)+T_{qt}(u;x_0,\rho_0/2) \Big]^{\beta(q-p)+1},
	\end{align*}
 where $\beta\in\mathbb{N}$ depends only on $N,p,s$ and $\alpha$.
	\end{Lemma}
Concerning the case when the coefficients need not be locally translation invariant, we have the following approximation lemma. 
\begin{Lemma}\label{approximation argument}
For any $\epsilon>0$, there exists a small $\delta=\delta(\mathsf{data},\epsilon)>0$ such that for any weak solution $u$ to \eqref{probA} in $B_{4}\equiv B_{4}(0)$ with 
\begin{equation*}
    \sup_{B_{4}}|u|\leq 1,\quad T_{ps}(u;0,4)+T_{qt}(u;0,4)\leq 1
\end{equation*}
and 
\begin{equation*}
    \left(\Xint-_{B_{4}}|f|^{\gamma}dx\right)^{1/\gamma}+
    \Xint-_{B_{4}}\Xint-_{B_{4}}(|a(x,y)-(a)_{4,0}|+|b(x,y)-(b)_{4,0}|)dxdy\leq \delta,
\end{equation*}
there exists a weak solution $v$ to 
\begin{align}\label{eq aux comp}
   \left\{
\begin{alignedat}{3}
\mathcal{L}_{\Tilde{a},\Tilde{b}}v &=0 &&\qquad \mbox{in  $B_{2}$}, \\
v&=u&&\qquad  \mbox{in  $\mathbb{R}^N\setminus B_{2}$},
\end{alignedat} \right. 
\end{align}
such that 
\begin{equation*}
    \|u-v\|_{L^{\infty}(B_{1})}\leq \epsilon,
\end{equation*}
where 
\begin{equation}
\label{tilde opr}
\begin{aligned}
\Tilde{a}(x,y)=\begin{cases}
(a)_{4,0}\quad&\text{if }(x,y)\in B_{4}\times B_{4},\\
a(x,y)\quad&\text{otherwise},
\end{cases}
\quad\mbox{and}\quad
\Tilde{b}(x,y)=\begin{cases}
(b)_{4,0}\quad&\text{if }(x,y)\in B_{4}\times B_{4},\\
b(x,y)\quad&\text{otherwise}.
\end{cases}
\end{aligned}
\end{equation}
\end{Lemma}
\begin{proof}
The existence of a weak solution $v$ to \eqref{eq aux comp} is given by Theorem \ref{thmexst}, below. To prove the claim, we proceed by the method of contradiction. Suppose there exist $\epsilon_{0}>0$ and sequences $\{a_{k}\}_{k\in\mathbb{N}}$, $\{b_{k}\}_{k\in\mathbb{N}}$, $\{f_{k}\}_{k\in\mathbb{N}}$ and $\{u_{k}\}_{k\in\mathbb{N}}$ such that 
\begin{equation}
\label{iteration pde}
    \mathcal{L}_{a_{k},b_{k}}u_{k}=f_k\quad\text{in }B_{4}
\end{equation}
with 
\begin{equation}
\label{asump for itera}
    \sup_{B_{4}}|u_{k}|\leq 1,\quad T_{ps}(u_{k};0,4)+T_{qt}(u_{k};0,4)\leq 1
\end{equation}
and 
\begin{equation*}
    \left(\Xint-_{B_{4}}|f_{k}|^{\gamma}dx\right)^{1/\gamma}+\Xint-_{B_{4}}\Xint-_{B_{4}}(|a_{k}(x,y)-(a_{k})_{4,0}|+|b_{k}(x,y)-(b_{k})_{4,0}|)dxdy\leq \frac{1}{k},
\end{equation*}
but for any weak solution $v_{k}$ to 
\begin{align}
\label{limiting pde iteration}
   \left\{
\begin{alignedat}{3}
\mathcal{L}_{\Tilde{a}_{k},\Tilde{b}_{k}}v_{k} &=0 &&\qquad \mbox{in  $B_{2}$}, \\
v_k&=u_{k}&&\qquad  \mbox{in  $\mathbb{R}^N\setminus B_{2}$},
\end{alignedat} \right. 
\end{align}
there holds
\begin{equation}
\label{contr 1}
    \|u_{k}-v_{k}\|_{L^{\infty}(B_{1})}>\epsilon_{0}.
\end{equation}
Set $w_{k}:=u_{k}-v_{k}$. Then, from Lemma \ref{boundedness1} and Lemma \ref{boundedness2}, we see that $v_{k}\in L^{\infty}(B_{2})$ and hence $w_{k}\in L^{\infty}(B_{4})$. On account of \cite[Lemma 5.1]{JDSPDPP}, we check that $w_{k}$ is a well defined test function to the weak formulation of problems \eqref{iteration pde} and \eqref{limiting pde iteration}.
We next claim that 
\begin{align}\label{limit of wk}
		 \int_{B_{3/2}} |w_k(x)|^{p^*_s}dx\to 0 \quad\mbox{as }k\to\infty.
	\end{align} 
Testing $w_{k}$ to \eqref{iteration pde} and \eqref{limiting pde iteration}, we see that
\begin{equation*}
\begin{aligned}
I_{0}&:=\int_{\mathbb R^N}\int_{\mathbb R^N} \Tilde{a}_k(x,y) ([u_k(x)-u_k(y)]^{p-1}-[v_{k}(x)-v_{k}(y)]^{p-1})(w_k(x)-w_k(y))d\mu_1\\   
	 &\quad + \int_{\mathbb R^N}\int_{\mathbb R^N} \Tilde{b}_k(x,y) \Big([u_k(x)-u_k(y)]^{q-1}- [v_{k}(x)-v_{k}(y)]^{q-1}\Big)(w_k(x)-w_k(y))d\mu_2\\
	 &= \int_{\mathbb R^N}\int_{\mathbb R^N} (\Tilde{a}_{k}(x,y)-a_k(x,y)) [u_{k}(x)-u_{k}(y)]^{p-1}(w_k(x)-w_k(y))d\mu_1\\
	 &\quad + \int_{\mathbb R^N}\int_{\mathbb R^N} (\Tilde{b}_{k}(x,y)-b_k(x,y)) [u_{k}(x)-u_{k}(y)]^{q-1}(w_k(x)-w_k(y))d\mu_2
	 \\
	 &\quad+\int_{B_{4}}f_kw_{k}dx
	 =:I_{1}+I_{2}.
\end{aligned}
\end{equation*}
Now we estimate each $I_{i}$ for $i=0,1,2$ and $3$.\\
\textbf{Estimate of $I_{0}$}. Using \eqref{eqmon}, we see that
\begin{equation*}
    I_{0}\geq \frac{1}{\Lambda}[w_{k}]_{W^{s,p}(\mathbb{R}^{N})}^{p}.
\end{equation*}
\textbf{Estimate of $I_{2}$}. We first note that there exists a constant $c=c(q)$ such that
\begin{equation*}
\begin{aligned}
I_{2}&=\int_{B_{4}}\int_{B_{4}}(\Tilde{b}_{k}(x,y)-b_k(x,y)) [u_{k}(x)-u_{k}(y)]^{q-1}(w_k(x)-w_k(y))d\mu_2\\
&\leq c\int_{B_{4}}\int_{B_{4}}|\Tilde{b}_{k}(x,y)-b_k(x,y)| |u_{k}(x)-u_{k}(y)|^{p-1}|w_k(x)-w_k(y)|d\mu_1,
\end{aligned}
\end{equation*}
where we have used the fact that $tq\leq ps$ and \eqref{asump for itera}.
In addition, using H\"older's inequality, \eqref{eqbdB}, Theorem \ref{Main theorem1} and Young's inequality, we find that there is a constant $c=c(\mathsf{data})$ which is independent of $k$ such that
\begin{equation*}
\begin{aligned}
I_{2} &\leq c \left(\Xint-_{B_{4}}\Xint-_{B_{4}}|\Tilde{b}_{k}(x,y)-b_k(x,y)||u_{k}(x)-u_{k}(y)|^{p}d\mu_{1}\right)^{(p-1)/p}[w_{k}]_{W^{s,p}(B_{4})}\\
&\leq c \left(\Xint-_{B_{4}}\Xint-_{B_{4}}|\Tilde{b}_{k}(x,y)-b_k(x,y)|dxdy\right)^{\frac{\sigma(p-1)}{p(1+\sigma)}}\left(\Xint-_{B_{4}}\Xint-_{B_{4}}\left(\frac{|u_k(x)-u_k(y)|^{p}}{|x-y|^{N+ps}}\right)^{(1+\sigma)}dxdy\right)^{\frac{(p-1)}{p(1+\sigma)}}\\
&\quad\times[w_{k}]_{W^{s,p}(B_{4})}\\
&\leq c\left(\frac{1}{k}\right)^{\frac{\sigma}{1+\sigma}}+\frac{I_{0}}{16},
\end{aligned}
\end{equation*}
where we have chosen a sufficiently small $\sigma>0$ so that Theorem \ref{Main theorem1} holds.
Likewise, we have
\begin{equation*}
    I_{1}\leq c\left(\frac{1}{k}\right)^{\frac{\sigma}{1+\sigma}}+\frac{I_{0}}{16}.
\end{equation*}
\textbf{Estimate of $I_{3}$.} We use H\"older's inequality, Young's inequality and the Sobolev-Poincar\'e inequality to discover that
\begin{equation*}
\begin{aligned}
    I_{3}\leq \|f_k\|_{L^{\gamma}(B_{4})}\|w_{k}\|_{L^{\gamma'}(B_{4})}
    \leq c\|f_k\|_{L^{\gamma}(B_{4})}^{p'}+\frac{I_{0}}{16}.
\end{aligned}    
\end{equation*}
Combining all estimates $I_{0},I_{1}$, and $I_{2}$, we have
\begin{equation}
\label{limit of wk in fs}
    [w_{k}]_{W^{s,p}(\mathbb{R}^{N})}^{p}\leq c\left(\frac{1}{k}\right)^{\frac{\sigma}{1+\sigma}},
\end{equation}
where $c$ is independent of $k$. Therefore the claim \eqref{limit of wk} follows by the Sobolev-Poincar\'e inequality and \eqref{limit of wk in fs}.
Moreover $u_{k}$ and $v_{k}$ are H\"older continuous in $B_{2}$ with uniform bound independent of $k$ as in \cite[Lemma 5.1]{JDSPDPP}. 
By the Arzela-Ascoli theorem, there is a function $w$ such that $w_{k}\to w$ in $C^{\beta}(B_{3/2})$,  up to a subsequence, for some $\beta\in(0,1)$. By the uniqueness of the limit together with \eqref{limit of wk}, we have that
\begin{equation*}
    \lim_{k\to\infty}\|u_{k}-v_{k}\|_{L^{\infty}(B_{3/2})}=0,
\end{equation*}
which is a contradiction to \eqref{contr 1}.
\qed
\end{proof}
\begin{Lemma}
\label{normalized sol ho}
Let $u$ be a  weak solution to \eqref{probA} in $B_{4}\equiv B_{4}(0)$ with 
\begin{equation}
\label{base assum}
    \sup_{B_{4}}|u|\leq 1\quad\text{and}\quad T_{ps}(u;0,4)+T_{qt}(u;0,4)\leq 1.
\end{equation}
Given $\alpha\in (0, \Theta)$, where $\Theta$ is given by \eqref{eqTheta}, there exists a small constant  $\delta=\delta(\mathsf{data},\alpha)>0$ such that if kernel coefficients $a$ and $b$ are $(\delta,4)$-vanishing in $B_{4}\times B_{4}$ and
\begin{equation*}
    \left(\Xint-_{B_{4}}|f|^{\gamma}dx\right)^{1/\gamma}\leq\delta,
\end{equation*} 
then $u\in C^{\alpha}(B_{1})$ with the estimate 
\begin{equation*}
    [u]_{C^{\alpha}(B_{1})}\leq c
\end{equation*}for some constant $c=c(\mathsf{data},\alpha)$.
\end{Lemma}
\begin{proof}
Let $\alpha\in(0,\Theta)$ be fixed. We now show that for any $x\in B_{1}$, there is a constant $A^{x}\in \mathbb{R}$ such that 
\begin{equation*}
    \sup_{y\in B_{r}(x)}|u(y)-A^{x}|\leq cr^{\alpha},
\end{equation*}
for any $r\in(0,1]$ and for some constant  $c=c(\mathsf{data},\alpha)$. Using translation argument as in \cite[Proposition 4.2]{NHH}, it suffices to prove the case for $x=0$. To this end, we show the following claim.\\
\textbf{Claim.} There exist $\rho=\rho(\mathsf{data},\alpha)\in(0,1/4)$ and a sequence $\{A_{k}\}_{k=-1}^{\infty}$ with $A_{-1}=0$ such that for all $k\geq 0$,
\begin{equation}
\label{ind 1}
    |A_{k}-A_{k-1}|\leq 2\rho^{(k-1)\alpha},\quad \sup_{B_{4}}|u(\rho^{k}x)-A_{k}|\leq\rho^{k\alpha}
\end{equation}
and
\begin{equation}
\label{ind 2}
    T_{ps}\left(\left(\frac{u(\rho^{k}x)-A_{k}}{\rho^{k\alpha}}\right);0,4\right)+T_{qt}\left(\left(\frac{u(\rho^{k}x)-A_{k}}{\rho^{k\alpha}}\right);0,4\right)\leq 1.
\end{equation}
To prove the claim, we take $\rho>0$ sufficiently small depending only on $\mathsf{data}$ and $\alpha>0$ such that
\begin{equation}
\label{rho condi}
    \rho^{\frac{\Theta-\alpha}{2}}\leq \frac{1}{12^{\Theta_{0}+2}c_{1}c_{2}}\min\left\{\left[sp-\left(\frac{\Theta_{0}+\alpha}{2}\right)\right]^{\frac{1}{p-1}},1\right\},
\end{equation}
where $c_{1}=c_{1}(\mathsf{data})\geq1$ and $c_{2}=c_{2}(\mathsf{data})\geq1$ are constants which are determined later.
For $k=0$, we take $A_{0}=0$, then \eqref{ind 1} and \eqref{ind 2} hold by \eqref{base assum}. Suppose that \eqref{ind 1} and \eqref{ind 2} hold for $k=0,1,\ldots i$. Set 
\begin{equation*}
    u_{i}(x)=\frac{u(\rho^{i}x)-A_{i}}{\rho^{\alpha i}},\quad f_{i}(x)=\rho^{\left(sp-\alpha(p-1)\right)i}f(\rho^{i}x),\quad x\in\mathbb{R}^{N};
\end{equation*} 
\begin{equation*}
    a_{i}(x,y)=a(\rho^{i}x,\rho^{i}y)\quad\text{and}\quad b_{i}(x,y)=b(\rho^{i}x,\rho^{i}y)\rho^{(sp-tq+\alpha(q-p))i},\quad (x,y)\in\mathbb{R}^{2N}.
\end{equation*}
Then $u_{i}$ is a  weak solution to
\begin{equation*}
    \mathcal{L}_{a_{i},b_{i}}u_{i}=f_{i},\quad\text{in }B_{4}.
\end{equation*}
By the inductive assumption, we have
\begin{equation*}
    \sup_{B_{4}}|u_{i}|\leq 1\quad\mbox{and}\quad T_{ps}(u_{i};0,4)+T_{qt}(u_{i};0,4)\leq 1.
\end{equation*}
Since $\rho<1$, we notice that
\begin{equation*}
    \Lambda^{-1}\leq a_{i}\leq \Lambda\quad\text{and}\quad 0\leq b_{i}\leq \Lambda.
\end{equation*}
By Lemma \ref{approximation argument}, we find $\delta_{0}=\delta_{0}(\mathsf{data},\epsilon)$ corresponding to the given 
\begin{equation}
\label{epsi}
    \epsilon=\frac{\rho^{\alpha}}{16c_{2}}.
\end{equation}
Taking $\delta=\frac{\delta_{0}}{3}$,  we see that $a_{i}$ and $b_{i}$ are $(\delta,4)$-vanishing in $B_{4}\times B_{4}$ because $a$ and $b$ are $(\delta,4)$-vanishing in $B_{4}\times B_{4}$. Therefore, we check that
\begin{equation*}
    \left(\Xint-_{B_{4}}|f_{i}|^{\gamma}dx\right)^{1/\gamma}+
    \Xint-_{B_{4}}\Xint-_{B_{4}}(|a_{i}(x,y)-(a_{i})_{4,0}|+|b_{i}(x,y)-(b_{i})_{4,0}|)dxdy\leq \delta_{0}.
\end{equation*}
By Lemma \ref{approximation argument}, there exists a weak solution $v_i$ to the following problem:
\begin{align*}
   \left\{
\begin{alignedat}{3}
\mathcal{L}_{\Tilde{a}_{i},\Tilde{b}_{i}}v_i &=0 &&\qquad \mbox{in  $B_{2}$}, \\
v_i&=u_{i}&&\qquad  \mbox{in  $\mathbb{R}^N\setminus B_{2}$},
\end{alignedat} \right. 
\end{align*}
such that 
\begin{equation}
\label{com ep}
    \|u_{i}-v_i\|_{L^{\infty}(B_{1})}\leq \epsilon,
\end{equation}
where $\Tilde{a}_{i}$ and $\Tilde{b}_{i}$ are defined as in \eqref{tilde opr} corresponding to $a_i$ and $b_i$, respectively. Before checking the assumptions \eqref{ind 1} and \eqref{ind 2}, we specify the constants $c_{1}$ and $c_{2}$.\\
%\noindent
\textbf{1. Constant $c_{1}$}. We first note that there is a $c=c(\mathsf{data})$ independent of $i$ such that
\begin{equation}
\label{lp of v}
    \|v_i\|_{L^{p_{s}^{*}}(B_{2})}\leq c,
\end{equation}
by following the proof in Lemma \ref{approximation argument} with \eqref{base assum}. From \eqref{lp of v} and \eqref{com ep}, we see that
\begin{equation}
\label{tps for i}
    T_{ps}(v_i;0,2\rho)\leq c\left(\|v_i\|_{L^{\infty}(B_{3/2})}+\|v_i\|_{L^{p}(B_{2})}T_{ps}(u;0,2)\right)\leq c,
\end{equation}
where $c=c(\mathsf{data})$. In light of Lemma \ref{thmingenk} and \eqref{tps for i}, there exists a constant $c_{1}=c_{1}(\mathsf{data})\geq1$ which is independent of $i$ such that
\begin{equation}
\label{seminorm Ho}
    [v_i]_{C^{\Tilde{\alpha}}(B_{1})}\leq c_{1},
\end{equation}
where $\Tilde{\alpha}=(\Theta_{0}+\alpha)/2<1$.

\noindent
\textbf{2. Constant $c_{2}$.}
Set 
\begin{equation}
\label{constant}
    c_{2}=\max\left\{1,T_{ps}(1;x_{0},R)+T_{qt}(1;x_{0},R)\right\},\quad \text{for } R>0\quad\text{and}\quad x_{0}\in\mathbb{R}^{N}.
\end{equation}
Then we find that $c_{2}=c_{2}(\mathsf{data})\geq1$ and it is independent of $R$ and $x_{0}$.\\
Let $A_{i+1}=A_{i}+\rho^{i\alpha}v_i(0)$. We now check the inductive assumptions \eqref{ind 1} and \eqref{ind 2} for $i=k+1$.
We first note that \eqref{com ep} also implies that
\begin{equation}
\label{ind i+1}
    |A_{i+1}-A_{i}|\leq \rho^{i\alpha}|v_i(0)|\leq 2\rho^{i\alpha}.
\end{equation}
In addition,  by \eqref{rho condi}, \eqref{com ep} and \eqref{seminorm Ho}, we see that
\begin{equation*}
\begin{aligned}
\sup_{B_{4}}|u(\rho^{i+1}x)-A_{i+1}|&=\sup_{B_{4\rho}}|u(\rho^{i}x)-A_{i}-\rho^{i\alpha}v_{i}(0)|\\
&\leq \rho^{i\alpha}\sup_{B_{4\rho}}|u_{i}(x)-v_{i}(x)|+\rho^{i\alpha}\sup_{B_{4\rho}}|v_{i}(x)-v_{i}(0)| \\
&\leq \frac{\rho^{(i+1)\alpha}}{16}+c_{1}(4\rho)^{\Tilde{\alpha}}\rho^{i\alpha} \leq \rho^{(i+1)\alpha},
\end{aligned}
\end{equation*}
where we have used $\rho\in\left(0,{1/4}\right)$.
Thus, we have shown that \eqref{ind 1} holds for $k=i+1$. Moreover,
we observe  that 
\begin{align}\label{estimate Jsp}
J_{s,p}:=&\left((4\rho^{i+1})^{sp}\int_{B_{\rho^{i}}\setminus B_{4\rho^{i+1}}}\frac{\left|u(x)-A_{i+1}\right|^{p-1}}{\rho^{(i+1)\alpha(p-1)}|x|^{N+sp}}dx\right)^{\frac{1}{p-1}}\nonumber\\
&\leq \left((4\rho)^{sp}\int_{B_{1}\setminus B_{4\rho}}\frac{\left|u(\rho^{i}x)-\left(A_{i}+v_i(x)\rho^{i\alpha}\right)\right|^{p-1}}{\rho^{(i+1)\alpha(p-1)}|x|^{N+sp}}dx\right)^{\frac{1}{p-1}}\nonumber\\
&\quad+\left((4\rho)^{sp}\int_{B_{1}\setminus B_{4\rho}}\frac{\left|v_i( x)-v_i(0)\right|^{p-1}}{\rho^{\alpha(p-1)}|x|^{N+sp}}dx\right)^{\frac{1}{p-1}} \nonumber\\
&\leq c_{2}\frac{\|u_{i}-v_{i}\|_{L^{\infty}(B_{1})}}{\rho^{\alpha}}+c_{1}\left((4\rho)^{sp}\int_{B_{1}\setminus B_{4\rho}}\frac{dx}{\rho^{\alpha(p-1)}|x|^{N+sp-\Tilde{\alpha}(p-1)}}\right)^{\frac{1}{p-1}} \nonumber\\
&\leq c_{2}\frac{\epsilon}{\rho^{\alpha}}+\frac{c_{2}c_{1}4^{\tilde\alpha}}{\left(sp-\Tilde{\alpha}(p-1)\right)^{\frac{1}{p-1}}}
\rho^{\frac{\Theta-\alpha}{2}}\leq \frac{1}{8},
\end{align}
where we have used \eqref{com ep}, \eqref{seminorm Ho}, \eqref{epsi} and \eqref{rho condi}. Similarly, we deduce that
\begin{equation}
\label{estimate Jtq}
    J_{t,q}\leq\frac{1}{8}.
\end{equation}
Consequently, using \eqref{estimate Jsp}, \eqref{estimate Jtq} and \eqref{constant}, we obtain
\begin{equation*}
\begin{aligned}
\sum_{l}T_{l}\left(\left(\frac{u(\rho^{i+1}x)-A_{i+1}}{\rho^{\alpha(i+1)}}\right);0,4\right)
&=\sum_{l}T_{l}\left(\left(\frac{u(x)-A_{i+1}}{\rho^{\alpha(i+1)}}\right);0,4\rho^{i+1}\right)\\
&\leq \sum_{l}(4\rho)^{\Theta_{0}}T_{l}\left(\left(\frac{u(x)-A_{i+1}}{\rho^{\alpha(i+1)}}\right);0,\rho^{i}\right)+J_{s,p}+J_{t,q}\\
&\leq \sum_{l}(4\rho)^{\Theta_{0}}T_{l}\left(\left(\frac{u(x)-A_{i+1}}{\rho^{\alpha(i+1)}}\right);0,\rho^{i}\right)+\frac{1}{4}
\end{aligned}
\end{equation*}
for $l\in\{ps,qt\}$.
With the help of \eqref{constant}, \eqref{ind i+1}, \eqref{ind 1} and \eqref{ind 2} for $k=i$, we further estimate
\begin{equation*}
\begin{aligned}
    &\sum_{l}(4\rho)^{\Theta_{0}}T_{l}\left(\left(\frac{u(x)-A_{i+1}}{\rho^{\alpha(i+1)}}\right);0,\rho^{i}\right) \\
    &\leq 4^{\Theta_{0}}\Bigg[ \sum_{l}\rho^{\Theta_{0}}T_{l}\left(\left(\frac{u(x)-A_{i}}{\rho^{\alpha(i+1)}}\right);0,4\rho^{i}\right)+\sum_{l}\rho^{\Theta_{0}}T_{l}\left(\frac{1}{\rho^{\alpha}};0,4\rho^{i}\right)\\
    &\quad\quad\quad+\sum_{l}\rho^{\Theta_{0}}T_{l}\left(\left(\frac{A_{i}-A_{i+1}}{\rho^{\alpha(i+1)}}\right);0,\rho^{i}\right)\Bigg]\\
    &\leq4^{\Theta_{0}}\left[ \sum_{l}\rho^{\Theta_{0}-\alpha}T_{l}\left(\left(\frac{u(x)-A_{i}}{\rho^{\alpha i}}\right);0,4\rho^{i}\right)+\rho^{\Theta_{0}-\alpha}3c_{2}\right]\\
    &\leq 4^{\Theta_{0}}\left[4c_{2}\rho^{\Theta_{0}-\alpha}\right]\leq \frac{1}{4}.
\end{aligned}    
\end{equation*}
It gives that \eqref{ind 2} holds for $k=i+1$, hence the \textbf{claim} follows. Thus, from the \textbf{claim} with simple computations (see \cite{caffcab}), we see that 
\begin{equation*}
    \lim_{i\to\infty}A_{i}=A<+\infty.
\end{equation*}
 In addition, for any $r\in(0,1]$, there is a constant $c=c(\mathsf{data},\alpha)$ such that
\begin{equation*}
\begin{aligned}
    \|u(x)-A\|_{L^{\infty}(B_{r})}=\|u(x)-A_{j}\|_{L^{\infty}(B_{r})}+|A-A_{j}|
    &\leq \rho^{j\alpha}+\sum_{k=j}^{\infty}2\rho^{k\alpha}\\
    &\leq c\rho^{j\alpha}\leq cr^{\alpha},
\end{aligned}    
\end{equation*}
where $j$ is the unique non-negative number satisfying $\rho^{j+1}<r\leq \rho^{j}$. 
\qed
\end{proof}

%%%%%%%%%%%%%%%%%%%%%%%%%%%%%%%%

\begin{Lemma}
\label{almost final H}
Let $u$ be a local weak solution to \eqref{probA} and let the functions $a$ and $b$ be in VMO. Then for any $\alpha\in(0,\Theta)$, $u\in C^{\alpha}_{\mathrm{loc}}(\Omega)$. 
\end{Lemma}
\begin{proof}
Let $\alpha\in(0,\Theta_{0})$ be fixed and let  $\delta=\delta(\mathsf{data},\alpha)$ be as obtained in  Lemma \ref{normalized sol ho}. Suppose $B_{\rho_{0}}(x_{0})\Subset\Omega$. It suffices to show that $u\in C^{\alpha}\left(\overline{B_{\rho_{0}}(x_{0})}\right)$. Set 
\begin{equation}
\label{distance of R}
    R:=\mathrm{dist}\left(B_{\rho_{0}}(x_{0}),\partial\Omega\right),\quad R_{0}:=\rho_{0}+R/2
\end{equation}
and
\begin{equation*}
\begin{aligned}
    \mathcal{M}:=&8c_{2}\Bigg[\|u\|_{L^{\infty}(B_{R_{0}}(x_{0}))}+T_{ps}(u;x_{0},R_{0})+T_{qt}(u;x_{0},R_{0})\\
    &\quad+\left(\frac{R_{0}^{sp-\frac{n}{\gamma}}\|f\|_{L^{\gamma}(B_{R_{0}}(x_{0}))}}{\delta}\right)^{\frac{1}{p-1}}+1\Bigg]\times \left(\frac{2R_{0}}{R}\right)^\frac{N+sp}{p-1},
\end{aligned}
\end{equation*}
where $c_{2}$ is given as in \eqref{constant} of Lemma \ref{normalized sol ho}.
Then we find that there is a constant
\begin{equation}
\label{condition of rho}
    \rho\in\left(0,\min\left\{\frac{\rho_{0}}{4},\frac{R}{4}\right\}\right)
\end{equation}
depending only on $\mathsf{data},\mathcal{M}, \nu_{a}$ and $\nu_{b}$ such that
\begin{equation*}
    \mathcal{M}^{q-p}\left(\frac{\rho}{4}\right)^{sp-tq}\leq 1
\end{equation*}
(this is possible because of the condition $ps>qt$) and the kernel coefficients $a$ and $b$ are $(\delta,\rho)$-vanishing in $B_{\rho_{0}}(x_{0})\times B_{\rho_{0}}(x_{0})$. We further note that 
\begin{equation*}
    B_{\rho}(z)\subset B_{R_{0}}(x_{0}),\quad \text{for every }z\in B_{\rho_{0}}(x_{0}).
\end{equation*}
We define for any $z\in B_{\rho_{0}}(x_{0})$, 
\begin{equation*}
    u_{z}(x)=\frac{u\left(\frac{\rho }{4}x+z\right)}{\mathcal{M}},\quad f_{z}(x)=\left(\frac{\rho}{4}\right)^{sp}\frac{1}{\mathcal{M}^{p-1}}f\left(\frac{\rho }{4}x+z\right),\quad x\in B_{4}
\end{equation*}
and
\begin{equation*}
    a_{z}(x,y)=a\left(\frac{\rho }{4}x+z,\frac{\rho }{4}y+z\right),\ b_{z}(x,y)=\mathcal{M}^{q-p}\left(\frac{\rho}{4}\right)^{sp-tq}b\left(\frac{\rho }{4}x+z,\frac{\rho }{4}y+z\right),\ (x,y)\in\mathbb{R}^{2N}.
\end{equation*}
Then we directly see that
\begin{equation*}
\mathcal{L}_{a_{z},b_{z}}u_{z}=f_{z},\quad \text{in }B_{4}(0)
\end{equation*}
with 
\begin{equation*}
    \sup_{B_{4}}|u_{z}|\leq 1
    \quad\text{and}\quad \left(\Xint-_{B_{4}}|f|^{\gamma}dx\right)^{1/\gamma}\leq\delta.
\end{equation*}
On the other hand, for $l\in\{ps,qt\}$, we note that
\begin{align*}
%\label{first tail}
\sum_{l}T_{l}(u_{z};0,4)=\frac{1}{\mathcal{M}}\sum_{l}T_{l}(u;z,\rho)
&\leq \left(\frac{2R_{0}}{R}\right)^\frac{N+sp}{p-1} \frac{1}{\mathcal{M}}\sum_{l}T_{l}(u;x_0,R_{0}) \nonumber\\
& \quad +\frac{1}{\mathcal{M}}\sum_{l}T_{l}\Big(\Big(\frac{R}{2R_{0}}\Big)^\frac{N+sp}{p-1}\frac{\mathcal{M}}{8c_{2}};z,\rho\Big) \nonumber\\
&\leq \frac{1}{8c_2} +\frac{1}{4}\leq 1,
\end{align*}
where we have used \eqref{condition of rho}, \eqref{distance of R} and the fact that
\begin{equation*}
    |y-z|\geq|y-x_{0}|-|x_{0}-z|\geq |y-x_{0}|-\frac{\rho_{0}}{R_{0}}|y-x_{0}|\geq \frac{R}{2R_{0}}|y-x_{0}|,\quad y\in B_{R_{0}}(x_{0})^c.
\end{equation*} 
Moreover, the kernel coefficients $a_{z}$ and $b_{z}$ are $(\delta,4)$-vanishing in $B_{4}\times B_{4}$ and the following holds:
\begin{align*}
    \Lambda^{-1}\leq a_z\leq \Lambda \quad\mbox{and}\quad 0\leq b_{z}\leq\Lambda.
\end{align*} 
By Lemma \ref{normalized sol ho}, $u_{z}\in C^{\alpha}(\overline{B_{1}})$. Scaling it back, we obtain  $u\in C^{\alpha}(\overline{B_{\rho}(z)})$ for any $z\in \overline{B_{\rho_{0}}(x_{0})}$. Using the standard covering argument as in \cite[Theorem 4.3]{NHH}, we have the desired result.
\qed
\end{proof}

Now we return to our original problem; that is, the coefficient function $a$ has the form $a(x,y,u(x),u(y))$, where $u$ is a solution under consideration. 
\begin{Lemma}\label{lem VMO u}
For a weak solution $u\in C^{\sigma}_{\mathrm{loc}}(\Omega)$ to  \eqref{probM}, for some $\sigma\in(0,1)$, the coefficient function $a(x,y,u(x),u(y))$ is in VMO on $B_{\rho}(x_{0})\times B_{\rho}(y_{0})$ for any $x_0,y_0\in\mathbb{R}^N$ and $\rho>0$ satisfying $B_{\rho}(x_{0}),B_{\rho}(y_{0})\Subset\Omega$.
	\end{Lemma}
\begin{proof}
Fix $x_{0},y_{0}\in \Omega$ and $\rho>0$ such that $B_{\rho}(x_{0}),B_{\rho}(y_{0})\Subset\Omega$. Then, for all $r<\rho$, using the continuity and VMO properties, we have
	\begin{align*}
	 &\Xint-_{B_{r}(x_{0})}\Xint-_{B_{r}(y_{0})}\bigg|a(x,y,u(x),u(y))-\Xint-_{B_{r}(x_{0})}\Xint-_{B_{r}(y_{0})}a(x',y',u(x'),u(y'))dx'dy'\bigg|dxdy \nonumber\\
	 &\leq 2\Xint-_{B_{r}(x_{0})}\Xint-_{B_{r}(y_{0})}|a(x,y,u(x),u(y))-a(x,y,u(x_0),u(y_0))|dxdy \nonumber\\
	 & \ + \Xint-_{B_{r}(x_{0})}\Xint-_{B_{r}(y_{0})} |a(x,y,u(x_0),u(y_0)) - (a)_{r,x_0,y_0}(u(x_0),u(y_0))|dxdy \nonumber\\
	 &\leq 2\Xint-_{B_{r}(x_{0})}\Xint-_{B_{r}(y_{0})} \omega_{a,M}\Big(\frac{|u(x)-u(x_0)|+|u(y)-u(y_0)|}{2}\Big) dxdy+\nu_{a,M}(\rho),
		\end{align*}
 where $M=2\max\{\|u\|_{L^\infty(B_\rho(x_0))}, \|u\|_{L^\infty(B_\rho(y_0))}\}$, $\omega_{a,M}$ is given by property (A3) and $\nu_{a,M}$ by \eqref{vmo modul}. The right-hand side terms converge to $0$ as $\rho\to 0$ due to the assumption (A3) and the VMO condition of the definition \ref{defVMO}. This proves the lemma.\QED
\end{proof}
	
\textbf{Proof of Theorem \ref{thmmain}}: 
Let $B_{\rho_{0}}(x_{0})\Subset\Omega$ and $\alpha\in(0,\Theta_{0})$ be fixed. It suffices to show that $u\in C^{\alpha}(\overline{B_{\rho_{0}}(x_{0})})$. Set 
\begin{equation*}
    R:=\mathrm{dist}\left(B_{\rho_{0}}(x_{0}),\partial\Omega\right)\quad\mbox{and}\quad R_{0}:=\rho_{0}+R/2.
\end{equation*}
In light of Lemma \ref{lem VMO u} and simple computations, we see that $A(x,y):=a(x,y,u(x),u(y))$ is in VMO on $B_{R_{0}}(x_0)\times B_{R_{0}}(x_0)$, symmetric and satisfies \eqref{eqbdA}. 
Since $u$ solves
\begin{equation*}
    \mathcal{L}_{A,b} \ u=f\quad\text{in }B_{R_{0}}(x_{0}),
\end{equation*}
 where $\Lambda^{-1}\leq A\leq \Lambda$, it gives that $u\in C^{\sigma}_{\mathrm{loc}}(\Omega)$ for some $\sigma=\sigma(\mathsf{data})\in(0,1)$.
 By Lemma \ref{almost final H} and Lemma \ref{lem VMO u}, we have the required one.
\QED

\section{The existence result}
This section provides the solvability of the following Dirichlet problem:		
	\begin{equation*}
		\left\{ \begin{array}{llr}
			\mathcal L_{a(\cdot,u),b} \; u=f &\mbox{in }\Omega, \\ 
			u=g &\mbox{in }\mathbb R^N\setminus\Omega,
		\end{array}
		\right. \tag{$\mathcal{G}$}\label{genPrb}
	\end{equation*}
 where $\Omega\subset\mathbb R^N$ is a bounded open set, $f$ and $g$ are suitable measurable functions. 
 With $g\in L^{p-1}_{ps}(\mathbb R^N)\cap L^{q-1}_{qt,b}(\mathbb R^N)$ and $\Omega\Subset\Omega'\Subset\mathbb R^N$, we define
	\begin{align*}%\label{bdryspace}
		X_{g,b}(\Omega,\Omega'):= \{ v\in \mathcal W_b(\Omega')\cap L^{p-1}_{ps}(\mathbb R^N)\cap L^{q-1}_{qt,b}(\mathbb R^N) : v=g \quad\mbox{a.e. in }\mathbb R^N\setminus\Omega \},
	\end{align*}
 equipped with the norm of $\mathcal W_b(\Omega')$. Once again, we will suppress the term $b$ from the above definition whenever it is clear in the context. Now we define the notion of a weak solution to \eqref{genPrb} as usual. 
 \begin{Definition}
 Let $f\in (\mathcal W(\Omega'))^*$ and $g\in \mathcal W(\Omega')\cap L^{p-1}_{ps}(\mathbb R^N)\cap L^{q-1}_{qt,b}(\mathbb R^N)$, for $\Omega\Subset\Omega'\Subset\mathbb R^N$.	A function $u\in X_g(\Omega,\Omega')$ is said to be a weak solution of the problem	\eqref{genPrb}, if for all $\phi\in X_0(\Omega,\Omega')$, 
	\begin{align*}%\label{solndef}
	  &\int_{\mathbb R^N}\int_{\mathbb R^N} \Bigg(a(x,y,u(x),u(y)) \frac{[u(x)-u(y)]^{p-1}}{|x-y|^{N+ps}}+b(x,y) \frac{[u(x)-u(y)]^{q-1}}{|x-y|^{N+qt}}\Bigg)(\phi(x)-\phi(y))dxdy \\
	   &= \langle f, \phi \rangle_{\mathcal W, \mathcal W^*}.
		\end{align*}
 \end{Definition}
To prove our existence result, we consider the case when the kernel coefficient $a(\cdot,\cdot,\cdot,\cdot)$ satisfies a global uniform continuity condition (stronger than (A3)), namely 
 \begin{itemize}
  \item[(A3)'] the function $a$ is uniformly continuous in $\mathbb{R}^N\times\mathbb{R}^N\times\mathbb{R}\times\mathbb{R}$; that is,  there is a non-decreasing function $\omega_{a}:[0,\infty)\rightarrow[0,\infty)$ with $\omega_{a}(0)=0$ and $\lim\limits_{t\downarrow 0}\omega_{a}(t)=0$ such that
		\begin{align}\label{eq140}
			|a(x,y,w,z)-a(x,y,w',z')|\leq\omega_{a}\Big(\frac{|w-w'|+|z-z'|}{2}\Big)
		\end{align}
		for all $z,z',w,w'\in\mathbb{R}$ uniformly in $(x,y)\in\mathbb{R}^N\times\mathbb{R}^N$.
 \end{itemize}
	
 \begin{Theorem}\label{thmexst}
 Suppose that $2\leq p\leq q<\infty$, $s,t\in (0,1)$ and that the coefficients satisfy the assumptions (A1), (A2) and (A3)'. Let $f\in (\mathcal W(\Omega'))^*$ and $g\in \mathcal W(\Omega')\cap L^{p-1}_{ps}(\mathbb R^N)\cap L^{q-1}_{qt,b}(\mathbb R^N)$, for $\Omega\Subset\Omega'\Subset\mathbb R^N$. Then, there exists a weak solution $u\in X_g(\Omega,\Omega')$ to the problem \eqref{genPrb}. In particular, if $g\in \mathcal W(\Omega')\cap L^{p-1}_{ps}(\mathbb R^N)\cap L^{q-1}_{qt}(\mathbb R^N)$ and $q\leq p^{*}_{s}$, then 
 \begin{equation*}
     u\in \mathcal W(\Omega')\cap L^{p-1}_{ps}(\mathbb R^N)\cap L^{q-1}_{qt}(\mathbb R^N).
 \end{equation*}
\end{Theorem}
 \begin{proof}
We see that as in the proof of \cite[Lemma 2.11]{brasco}, we find that the space $X_0(\Omega,\Omega')$ is continuously embedded into $\mathcal{W}(\Omega')$. Moreover, we can directly verify that $X_0(\Omega,\Omega')$ is a separable uniformly convex Banach space. We now define a functional $\mathcal A: X_{0}(\Omega,\Omega')\to (\mathcal W(\Omega'))^*$ by  \begin{align*}
	\mathcal A:= \mathcal A_p+\mathcal A_q, 
 \end{align*}
	where 
	\begin{align*}
	 \langle \mathcal A_p(v),\phi\rangle &=
	  \int_{\Omega'}\int_{\Omega'} a(x,y,v(x),v(y)) \frac{[v(x)+g(x)-v(y)-g(y)]^{p-1}}{|x-y|^{N+ps}}(\phi(x)-\phi(y))dxdy \nonumber\\
	  &\quad+2 \int_{\mathbb R^N\setminus\Omega'}\int_{\Omega} a(x,y,v(x),g(y)) \frac{[v(x)+g(x)-g(y)]^{p-1}}{|x-y|^{N+ps}}\phi(x)dxdy \nonumber\\
	  &=:\langle \mathcal A^1_p(v),\phi\rangle + \langle \mathcal A^2_p(v),\phi\rangle \quad\mbox{for all }\phi\in \mathcal{W}(\Omega')
 \end{align*}
	and $\mathcal A_q$ is defined analogously.
 By virtue of H\"older's inequality and recalling the definition of $W$ (as stated in \eqref{eqw}), we obtain
	\begin{align}\label{eq121}
	  &|\langle \mathcal  A_q(v),\phi\rangle| \nonumber\\ %%check the line
	  &\leq \int_{\Omega'}\int_{\Omega'} b(x,y) \frac{|v(x)+g(x)-v(y)-g(y)|^{q-1}}{|x-y|^{N+qt}}|\phi(x)-\phi(y)|dxdy \nonumber\\
	  &\quad +c(q)\int_{\mathbb R^N\setminus\Omega'} \int_{\Omega} b(x,y) \frac{|v(x)+g(x)|^{q-1}+|g(y)|^{q-1}}{|x-y|^{N+qt}}|\phi(x)|dxdy  \nonumber\\
	  &\leq c \big([v]^{q-1}_{W^{t,q}_{b}(\Omega')} + [g]^{q-1}_{W^{t,q}_{b}(\Omega')} \big) [\phi]_{W^{t,q}_{b}(\Omega')}
	  + c\int_{\Omega'}W(x)|v(x)+g(x)|^{q-1}|\phi(x)|dx\nonumber\\
	  &\quad +c\int_{\Omega}|\phi(x)|\int_{\mathbb R^N\setminus\Omega'} b(x,y)\frac{|g(y)|^{q-1}}{|x-y|^{N+qt}}dydx \nonumber\\
	  &\leq c \big([v]^{q-1}_{W^{t,q}_{b}(\Omega')} + [g]^{q-1}_{W^{t,q}_{b}(\Omega')} \big) [\phi]_{W^{t,q}_{b}(\Omega')} \nonumber\\
	  &\quad+c\Big(\int_{\Omega'}W(x)|(v+g)(x)|^{q}dx \Big)^\frac{1}{q'} \Big(\int_{\Omega'}W(x)|\phi(x)|^{q}dx \Big)^\frac{1}{q} \nonumber \\ 
	   &\quad+ c\Big(\int_{\Omega}|\phi(x)|^{p}dx \Big)^\frac{1}{p}   \left(\sup_{x\in\mathbb{R}^{N}}\int_{\mathbb{R}^{N}}b(x,y)\frac{|g(y)|^{q-1}}{(1+|y|)^{N+tq}}dy\right)\nonumber\\
	  &\leq c\big(\|v\|^{q-1}_{\mathcal W(\Omega')}+\|g\|^{q-1}_{\mathcal W(\Omega')}+ \|g\|^{q-1}_{L^{q-1}_{qt,b}(\mathbb{R}^{N})}\big) \|\phi\|_{\mathcal W(\Omega')},
		\end{align}
 where $c=c(\mathsf{data}_2,\mathrm{dist}(\Omega,\Omega'))$.   A similar result holds for the $p$-term, too, by simply using the bound \eqref{eqbdA}. 
 Consequently, we get that $\mathcal A$ is a well defined operator.
 Moreover, \eqref{eq121} together with its $p$-counterpart shows that the operator $\mathcal{A}$ is bounded; i.e., it maps bounded sets to bounded sets. We next prove that $\mathcal{A}$ is weakly continuous. For this, let $\{u_k\}\subset X_0(\Omega,\Omega')$ be a sequence such that $u_k\rightharpoonup u$, weakly in $\mathcal{W}(\Omega')$ for some $u\in X_0(\Omega,\Omega')$. Then, we claim that
	\begin{align*}
	 \lim_{k\to\infty} \langle \mathcal{A}(u_k), \phi \rangle = \langle \mathcal A (u), \phi \rangle \quad\mbox{for all }\phi\in X_0(\Omega,\Omega'). 
	 \end{align*}
 Using the bound on the function $a$, we observe that 
	\begin{align}\label{eq exst cont est}
	  &|\langle \mathcal{A}(u_k)-\mathcal{A}(u), \phi \rangle| \nonumber\\
	  &\leq  \int_{\Omega'}\int_{\Omega'}  \big| a(x,y,u_k(x),u_k(y))-a(x,y,u(x),u(y))\big| |(u+g)(x)-(u+g)(y)|^{p-1} \nonumber\\ &\qquad\qquad\times|\phi(x)-\phi(y)|d\mu_{1} \nonumber\\
	  &\ + 2\int_{\Omega}\int_{\mathbb R^N\setminus\Omega'}  \big| a(x,y,u_k(x),g(y))-a(x,y,u(x),g(y))\big| |u(x)+g(x)-g(y)|^{p-1}  |\phi(x)|d\mu_{1} \nonumber\\
	  & \ + \Lambda
	  \int_{\Omega'}\int_{\Omega'}  \big| [(u_k+g)(x)-(u_k+g)(y)]^{p-1}-[(u+g)(x)-(u+g)(y)]^{p-1} \big| |\phi(x)-\phi(y)|d\mu_{1} \nonumber\\
	  & \ + 2\Lambda
	  \int_{\Omega}\int_{\mathbb R^N\setminus\Omega'}  \big| [(u_k+g)(x)-g(y)]^{p-1}-[(u+g)(x)-g(y)]^{p-1} \big| |\phi(x)|d\mu_{1} \nonumber\\
	  & \ + |\langle \mathcal{A}_q^1(u_k)-\mathcal{A}_q^1(u), \phi \rangle|+ |\langle \mathcal{A}_q^2(u_k)-\mathcal{A}_q^2(u), \phi \rangle|.
		\end{align}
By the definition of $X_0(\Omega,\Omega')$ together with  the weak convergence and compactness of the Sobolev embedding, we infer that, up to a subsequence,  $u_k(x)\to u(x)$ a.e.  in $\Omega'$. Therefore, using the uniform continuity condition of \eqref{eq140}, we deduce that the first two terms on the right-hand side of \eqref{eq exst cont est} converge to $0$, as $k\to\infty$. 
To prove the convergence of the third term, on the contrary, we assume that there exist $\epsilon_0>0$ and a subsequence $\{u_k\}$ (up to relabelling) such that 
\begin{align}\label{eq 1weak fail}
   \int_{\Omega'}\int_{\Omega'}  \big| [(u_k+g)(x)-(u_k+g)(y)]^{p-1}-[(u+g)(x)-(u+g)(y)]^{p-1} \big| |\phi(x)-\phi(y)|d\mu_{1}\geq \epsilon_0. 
\end{align}
Since $\{u_k\}$ is bounded in $\mathcal{W}(\Omega')$, using the definition of the norm on $\mathcal{W}(\Omega')$, we observe that the sequence $\bigg\{\frac{[(u_k+g)(x)-(u_k+g)(y)]^{p-1}}{|x-y|^\frac{N+ps}{p'}}\bigg\}$ is bounded in $L^{p'}(\Omega'\times\Omega')$. Therefore, by the reflexivity of the space $L^{p'}$ and the pointwise convergence $u_k\to u$ a.e. in $\Omega'$, up to a subsequence (again up to relabelling), we get that 
 \begin{align*}
    \frac{[(u_k+g)(x)-(u_k+g)(y)]^{p-1}}{|x-y|^\frac{N+ps}{p'}} \rightharpoonup \frac{[(u+g)(x)-(u+g)(y)]^{p-1}}{|x-y|^\frac{N+ps}{p'}} \quad\mbox{weakly in } L^{p'}(\Omega'\times\Omega'),
 \end{align*}
 as $k\to\infty$. Owing to the fact $\frac{|\phi(x)-\phi(y)|}{|x-y|^\frac{N+ps}{p}}\in L^p(\Omega'\times\Omega')$ (due to  $\phi\in\mathcal{W}(\Omega')$), we find a contradiction to \eqref{eq 1weak fail}. Therefore, the third term on the right-hand side of \eqref{eq exst cont est} converges to $0$, as $k\to\infty$. Similarly, for the fifth term, we note that the sequence $\bigg\{b(x,y)^\frac{1}{q'}\frac{[(u_k+g)(x)-(u_k+g)(y)]^{q-1}}{|x-y|^\frac{N+qt}{q'}}\bigg\}$ is bounded in $L^{q'}(\Omega'\times\Omega')$. Therefore, by the reflexivity of the space $L^{q'}$ and proceeding as above by noting that $b(x,y)^\frac{1}{q}\frac{|\phi(x)-\phi(y)|}{|x-y|^\frac{N+qt}{q}}\in L^q(\Omega'\times\Omega')$, we get that the fifth term also converges to $0$. It remains to prove the convergence of the fourth and the sixth terms on the right-hand side of \eqref{eq exst cont est}. Using \eqref{eqKKP2} and H\"older's inequality, we deduce that
 \begin{align*}
     &\int_{\Omega}\int_{\mathbb R^N\setminus\Omega'} b(x,y)  \big| [(u_k+g)(x)-g(y)]^{q-1}-[(u+g)(x)-g(y)]^{q-1} \big| |\phi(x)|d\mu_{2} \\
     &\leq c\int_{\Omega'} |u_k(x)-u(x)|^{q-1}|\phi(x)| \int_{\mathbb R^N\setminus\Omega'} \frac{b(x,y)}{|x-y|^{N+qt}}dydx \\
     & \ \ + c  \int_{\Omega}\int_{\mathbb R^N\setminus\Omega'}b(x,y) |\phi(x)||u_k(x)-u(x)||u(x)+g(x)-g(y)|^{q-2}d\mu_2\\
     &\leq c\int_{\Omega'} W(x) |u_k(x)-u(x)|^{q-1}|\phi(x)|dx  + c \Big( \int_{\Omega'} W(x) |u_k(x)-u(x)|^{q-1}|\phi(x)|dx\Big)^\frac{1}{q-1} \\
     &\qquad\times\Big( \int_{\Omega} \int_{\mathbb R^N\setminus\Omega'}b(x,y)|u(x)+g(x)-g(y)|^{q-1}|\phi(x)|d\mu_2\Big)^\frac{q-2}{q-1},
 \end{align*}
where $W$ is as defined in \eqref{eqw} with $\Omega'$ in place of $\Omega$. From \eqref{eq121}, we see that the second quantity on the right-hand side of the second term is finite. Then,  recalling the definition of the norm on $\mathcal{W}(\Omega')$ and arguing as above (the case of fifth term), we get that the sixth term on right-hand side of \eqref{eq exst cont est} converges to $0$. Similarly, we see that the fourth term on the right-hand side of \eqref{eq exst cont est} tends to $0$. Hence, we prove the claim.
\par Next, to prove coercivity of the operator $\mathcal{A}$,  for any $v\in X_0(\Omega,\Omega')$, using H\"older's and Young's inequalities, we first see that
 \begin{align}\label{eq114}
 &\langle \mathcal A_p^1(v),v\rangle \nonumber\\
 &=
\int_{\Omega'}\int_{\Omega'} a(x,y,v(x),v(y))   \Big([(v+g)(x)-(v+g)(y)]^{p-1}-[g(x)-g(y)]^{p-1} \Big) (v(x)-v(y))d\mu_{1} \nonumber\\
& \ +\int_{\Omega'}\int_{\Omega'}   a(x,y,v(x),v(y)) [g(x)-g(y)]^{p-1}  (v(x)-v(y))d\mu_{1} \nonumber\\
&\geq \frac{1}{c} \int_{\Omega'}\int_{\Omega'} |v(x)-v(y)|^p d\mu_1 - c \int_{\Omega'}\int_{\Omega'}   |g(x)-g(y)|^{p-1}  |v(x)-v(y)|d\mu_{1} \nonumber\\
&\geq \frac{1}{c}\int_{\Omega'}\int_{\Omega'} |v(x)-v(y)|^p d\mu_1 - c \int_{\Omega'}\int_{\Omega'}   |g(x)-g(y)|^{p}  d\mu_{1},
 \end{align}
 where we have also used \eqref{eqbdA} and \eqref{eqmon}. Similarly, we discover that 
 \begin{align}\label{eq115}
 &\langle \mathcal A_q^1(v),v\rangle \geq  \frac{1}{c}\int_{\Omega'}\int_{\Omega'} b(x,y)|v(x)-v(y)|^q d\mu_2 - c\int_{\Omega'}\int_{\Omega'}b(x,y)   |g(x)-g(y)|^{q} d\mu_{2}.
 \end{align}
Furthermore, using the inequality \eqref{eqmon} once again and the definition of $W$, we observe that 
	\begin{align}\label{eq116}
	 \langle  \mathcal A_q^2(v),v\rangle 
	 &= \int_{\Omega}\int_{\mathbb R^N\setminus\Omega'} b(x,y) \Big([v(x)+g(x)-g(y)]^{q-1}-[g(x)-g(y)]^{q-1}\Big)v(x)d\mu_2 \nonumber\\
	 & \ + \int_{\Omega}\int_{\mathbb R^N\setminus\Omega'} b(x,y) [g(x)-g(y)]^{q-1}v(x)d\mu_2 \nonumber\\
	 &\geq \frac{1}{c} \int_{\Omega}\int_{\mathbb R^N\setminus\Omega'} |v(x)|^{q} \frac{b(x,y)}{|x-y|^{N+qt}}dxdy \nonumber\\
	 & \ - c \int_{\Omega}\int_{\mathbb R^N\setminus\Omega'}  (|g(x)|^{q-1}+|g(y)|^{q-1})|v(x)|\frac{b(x,y)}{|x-y|^{N+qt}}dxdy \nonumber\\
	 &\geq \frac{1}{c} \int_{\Omega'} W(x)|v(x)|^q dx-c\Big(\int_{\Omega'}W(x)|g(x)|^{q}dx \Big)^\frac{1}{q'} \Big(\int_{\Omega'}W(x)|v(x)|^{q}dx \Big)^\frac{1}{q} \nonumber\\ 
	& \ - c \Big(\int_{\Omega}|v(x)|dx \Big)  \|g\|^{q-1}_{L^{q-1}_{qt,b}(\mathbb{R}^{N})} \nonumber\\
	&\geq \frac{1}{c} \int_{\Omega'} W(x)|v(x)|^q dx-c\int_{\Omega'}W(x)|g(x)|^{q}dx -\epsilon\int_{\Omega}|v(x)|^pdx\nonumber\\
	&\quad
	 - c\epsilon^{\frac{-1}{p-1}} \|g\|^{p'(q-1)}_{L^{q-1}_{qt,b}(\mathbb{R}^{N})},
		\end{align}
 where we have also used the fact that $v=0$ in $\Omega'\setminus\Omega$ and Young's inequality on the last line. On a similar note,
 \begin{align}\label{eq117}
 \langle  \mathcal A_p^2(v),v\rangle 
	 &\geq \frac{1}{2c} \int_{\Omega'} |v(x)|^p dx -c\int_{\Omega'}|g(x)|^{p}dx 
 -  c \|g\|^{p}_{L^{p-1}_{ps}(\mathbb{R}^{N})}.
 \end{align}
Finally, combining \eqref{eq114}, \eqref{eq115}, \eqref{eq116} and \eqref{eq117} with $\epsilon=\frac{1}{4c}$, and recalling  the definition of the norm on $\mathcal{W}(\Omega')$, we obtain
	\begin{align*}%\label{eq118}
	 \langle  \mathcal A(v),v\rangle &\geq \frac{1}{4c} \min\{ \|v\|_{\mathcal{W}(\Omega')}^{p}, \|v\|_{\mathcal{W}(\Omega')}^{q} \} -c \|g\|_{W^{s,p}(\Omega')}^{p}-c\|g\|_{W^{t,q}_{b}(\Omega')}^{q} \nonumber\\
	 & \ -c (\|g\|^{p}_{L^{p-1}_{ps}(\mathbb{R}^{N})}+\|g\|^{p'(q-1)}_{L^{q-1}_{qt,b}(\mathbb{R}^{N})}),
	\end{align*}
 where the constant $c$ depends only on $\mathsf{data}$, $\Omega$ and $\Omega'$. This proves the coercivity of the operator $\mathcal{A}$. \par 
 Consequently, by  \cite[Chap. II, Example 2.A, p. 40]{showalter}, it follows that the operator $\mathcal{A}$ is of $M$-type. Note that $X_0(\Omega,\Omega')$ is a separable reflexive Banach space and $(\mathcal{W}(\Omega'))^*\subset(X_0(\Omega,\Omega'))^*$. Hence, using \cite[Chap. II, Corollary 2.2, p. 39]{showalter}, we get that the map $\mathcal{A}$ is surjective. Moreover, the last statement is true considering $u\in L^{q-1}(\Omega)$ by $q\leq p_{s}^{*}$. This completes the proof of the theorem. \QED 
\end{proof}

\appendix
\section{Boundedness results}
We first give a boundedness result for the problem \eqref{probM} whose proof runs along the same lines of \cite[Proposition 3.1]{GKS2} by using the Caccioppoli estimate of Lemma \ref{Caccio estimate}. 
\begin{Lemma}
\label{boundedness1}
Suppose that  $q< p^*_{s}$ and $qt\leq ps$.  Let $u$ be a local weak solution to the problem \eqref{probM} in $\Omega$.  Then,  there exists a constant $c$ depending only on $\mathsf{data}$ and $\gamma$ (if $\gamma<\infty$) such that
\begin{equation*}
\begin{aligned}
\|u\|_{L^{\infty}(B_{r}(x_{0}))}\leq& c\Bigg(\left(\Xint-_{B_{2r}(x_{0})}|u(x)|^{\vartheta}dx\right)^{q/(p\vartheta)}+\|f\|^{1/(p-1)}_{L^{\gamma}(B_{2r}(x_{0}))}\\
&\quad+T_{ps}(u;x_0,2r)+T_{qt,b}(u;x_0,2r)+1\Bigg),
\end{aligned}
\end{equation*}
provided $B_{2r}(x_{0})\Subset \Omega$, where $\vartheta=\max\{q,p\varsigma\}$ with $\varsigma=\frac{p\gamma-(p^*_{s})'}{p(\gamma-(p^*_{s})')} \Big(<\frac{p^*_{s}}{p}\Big)$ if $\gamma<\infty$, while $\varsigma=1$ if $\gamma=\infty$.
\end{Lemma}
Let $B_{\rho_{0}}\equiv B_{\rho_{0}}(0)$.
We next consider the following problem: 
\begin{equation*}
		\left\{ \begin{array}{llr}
		\mathcal L_{a(\cdot,v),b}v=f &\mbox{in }B_{3\rho_0/2}, \\ 
		v=g &\mbox{in }\mathbb R^N\setminus B_{3\rho_0/2},
			\end{array}
		\right. \tag{$\mathcal{P}_b$}\label{probb}
		\end{equation*}
where $g\in \mathcal{W}_b(B_{2\rho_0})\cap  L^{\infty}(B_{2\rho_{0}})\cap L^{p-1}_{ps}(\mathbb R^N)\cap L^{q-1}_{qt,b}(\mathbb R^N)$ and $f\in L_{\mathrm{loc}}^{\gamma}(B_{2\rho_{0}})$ with $\gamma>\max\{1,{N/(ps)}\}$. Let $v\in X_{g,b}(B_{3\rho_{0}/2},B_{2\rho_{0}})$ be a weak solution to the problem \eqref{probb}. Then, 
 $v$ enjoys the same Caccioppoli-type estimate as in Lemma \ref{Caccio estimate} and hence Lemma \ref{boundedness1} holds for $v$, too. 
We next see the boundary estimate of the solution $v$. Precisely, we have the following estimate using \cite[Theorem 5]{KKP} and \cite[Proposition 3.1]{GKS2} with slight modifications.
\begin{Lemma}
\label{boundedness2}
Suppose that  $q< p^*_{s}$ and $qt\leq ps$.  Let $v\in X_{g,b}(B_{3\rho_{0}/2},B_{2\rho_{0}})$ be a weak solution to \eqref{probb} with $f\in L^{\gamma}(B_{2r}(x_{0})\cap B_{3\rho_{0}/2})$ and $g\in L^{\infty}(B_{2\rho_{0}})$, for some $r\in(0,1/16)$ and $x_{0}\in \partial B_{3\rho_{0}/2}$. Then there is a constant c depending only on $\mathsf{data}$ and $\gamma$ (if $\gamma<\infty$)  such that 
\begin{equation*}
\begin{aligned}
\|v\|_{L^{\infty}(B_{r}(x_{0}))}\leq& c\Bigg(\left(\Xint-_{B_{2r}(x_{0})}|v(x)|^{\vartheta}dx\right)^{q/(p\vartheta)}+\|f\|^{1/(p-1)}_{L^{\gamma}(B_{2r}(x_{0})\cap B_{3\rho_{0}/2})}\\
&\quad+T_{ps}(v;x_0,2r)+T_{qt,b}(v;x_0,2r)+\|g\|_{L^{\infty}(B_{2\rho_{0}})}+1\Bigg),
\end{aligned}
\end{equation*}
where $\vartheta$ is the same number as in Lemma \ref{boundedness1}.
\end{Lemma}

\medskip
\providecommand{\bysame}{\leavevmode\hbox to3em{\hrulefill}\thinspace}
\providecommand{\MR}{\relax\ifhmode\unskip\space\fi MR }
\providecommand{\MRhref}[2]{%
	\href{http://www.ams.org/mathscinet-getitem?mr=#1}{#2}
}
\providecommand{\href}[2]{#2}


\begin{thebibliography}{10}
	
	\bibitem{ABES}
	Pascal Auscher, Simon Bortz, Moritz Egert, and Olli Saari, \emph{Nonlocal
		self-improving properties: a functional analytic approach}, Tunis. J. Math.
	\textbf{1} (2019), no.~2, 151--183.
	
	\bibitem{baroni}
	Paolo Baroni, Maria Colombo, and Giuseppe Mingione, \emph{Regularity for
		general functionals with double phase}, Calculus of Variations and Partial
	Differential Equations \textbf{57} (2018), no.~2, 1--48.
	
	\bibitem{BL}
	Lorenzo Brasco and Erik Lindgren, \emph{Higher {S}obolev regularity for the
		fractional {$p$}-{L}aplace equation in the superquadratic case}, Adv. Math.
	\textbf{304} (2017), 300--354.
	
	\bibitem{brasco}
	Lorenzo Brasco, Erik Lindgren, and Armin Schikorra, \emph{Higher {H}\"{o}lder
		regularity for the fractional {$p$}-{L}aplacian in the superquadratic case},
	Adv. Math. \textbf{338} (2018), 782--846.
	
	\bibitem{BKO}
	Sun-Sig Byun, Hyojin Kim, and Jihoon Ok, \emph{Local H\"older continuity for
		fractional nonlocal equations with general growth}, Mathematische Annalen (2022), 1--40, https://doi.org/10.1007/s00208-022-02472-y.
	
	\bibitem{BOS}
	Sun-Sig Byun, Jihoon Ok, and Kyeong Song, \emph{H\"older regularity for weak
		solutions to nonlocal double phase problems}, Journal de Math{\'e}matiques
  Pures et Appliqu{\'e}es \textbf{168} (2022), 110--142.
	
	\bibitem{BPS}
	Sun-Sig Byun, Dian~K. Palagachev, and Pilsoo Shin, \emph{Global continuity of
		solutions to quasilinear equations with {M}orrey data}, C. R. Math. Acad.
	Sci. Paris \textbf{353} (2015), no.~8, 717--721.
	
	\bibitem{caffcab} Luis A. Caffarelli and Xavier Cabr\'{e}, \emph{Fully nonlinear elliptic equations},
	{American Mathematical Society Colloquium Publications},
	\textbf{43}, {American Mathematical Society, Providence, RI},
 {1995}.

	
	\bibitem{caff-silv}
	Luis Caffarelli and Luis Silvestre, \emph{Regularity results for nonlocal
		equations by approximation}, Arch. Ration. Mech. Anal. \textbf{200} (2011),
	no.~1, 59--88.
	
	\bibitem{CSP}
	Luis~A. Caffarelli and Pablo~Ra\'{u}l Stinga, \emph{Fractional elliptic
		equations, {C}accioppoli estimates and regularity}, Ann. Inst. H.
	Poincar\'{e} C Anal. Non Lin\'{e}aire \textbf{33} (2016), no.~3, 767--807.
	
	\bibitem{chaker}
	Jamil Chaker, Minhyun Kim, and Marvin Weidner, \emph{Regularity for nonlocal
		problems with non-standard growth}, Calc. Var. Partial Differential Equations \textbf{61} (2022),
	no.~6, Paper No. 227, 31.
	
	\bibitem{colombo2}
	Maria Colombo and Giuseppe Mingione, \emph{Bounded minimisers of double phase
		variational integrals}, Arch. Ration. Mech. Anal. \textbf{218} (2015), no.~1,
	219--273.
	
	\bibitem{colombo}
	\bysame, \emph{Regularity for double phase variational problems}, Arch. Ration.
	Mech. Anal. \textbf{215} (2015), no.~2, 443--496.
	
	\bibitem{Cr}
	Matteo Cozzi, \emph{Regularity results and {H}arnack inequalities for
		minimizers and solutions of nonlocal problems: a unified approach via
		fractional {D}e {G}iorgi classes}, J. Funct. Anal. \textbf{272} (2017),
	no.~11, 4762--4837.
	
	\bibitem{defilip-Ming}
	Cristiana De~Filippis and Giuseppe Mingione, \emph{Lipschitz bounds and
		nonautonomous integrals}, Arch. Ration. Mech. Anal. \textbf{242} (2021),
	no.~2, 973--1057. 
	
	\bibitem{defilip-ming-mix}
	\bysame, \emph{Gradient regularity in mixed
  local and nonlocal problems}, Mathematische Annalen (2022), 1--68, https://doi.org/10.1007/s00208-022-02512-7.

	\bibitem{filippis}
	Cristiana De~Filippis and Giampiero Palatucci, \emph{H\"{o}lder regularity for
		nonlocal double phase equations}, Journal of Differential Equations
	\textbf{267} (2019), no.~1, 547--586.
	
	\bibitem{CKPf}
	Agnese Di~Castro, Tuomo Kuusi, and Giampiero Palatucci, \emph{Local behavior of
		fractional {$p$}-minimizers}, Ann. Inst. H. Poincar\'{e} Anal. Non
	Lin\'{e}aire \textbf{33} (2016), no.~5, 1279--1299.
	
	\bibitem{Fr}
	Mouhamed~Moustapha Fall, \emph{Regularity results for nonlocal equations and
		applications}, Calc. Var. Partial Differential Equations \textbf{59} (2020),
	no.~5, Paper No. 181, 53.


	\bibitem{FMSY}
	Mouhamed~Moustapha Fall, Tadele Mengesha, Armin Schikorra, and Sasikarn Yeepo,
	\emph{Calder{\'o}n-zygmund theory for non-convolution type nonlocal equations
		with continuous coefficient}, Partial Differential Equations and Applications
	\textbf{3} (2022), no.~2, 1--27.
	
	\bibitem{zhang}
	Yuzhou Fang and Chao Zhang, \emph{On weak and viscosity solutions of nonlocal
		double phase equations}, International Mathematics Research Notices (2021),
	1--44,  https://doi.org/10.1093/imrn/rnab351.
	
	\bibitem{JDSPDPP}
	Jacques Giacomoni, Deepak Kumar, and Konijeti Sreenadh, \emph{Hölder
		regularity results for parabolic nonlocal double phase problems}, arXiv preprint
	arXiv:2112.04287v3 (2021).

	
	\bibitem{GKS}
	\bysame, \emph{Interior and boundary regularity results for strongly
		nonhomogeneous $p,q$-fractional problems}, Adv. Calc. Var. (2021), pp. 35, https://doi.org/10.1515/acv-2021-0040.
	
	\bibitem{GKS2}
	\bysame, \emph{Global regularity results for non-homogeneous growth fractional
		problems}, J. Geom. Anal. \textbf{32} (2022), no.~1, Paper No. 36, pp. 41.
	
	\bibitem{GiD}
	Enrico Giusti, \emph{Direct methods in the calculus of variations}, World
	Scientific Publishing Co., Inc., River Edge, NJ, 2003.
	
	\bibitem{KH1}
	Moritz Kassmann, \emph{A priori estimates for integro-differential operators
		with measurable kernels}, Calc. Var. Partial Differential Equations
	\textbf{34} (2009), no.~1, 1--21.
	
	\bibitem{KKP}
	Janne Korvenp\"{a}\"{a}, Tuomo Kuusi, and Giampiero Palatucci, \emph{The
		obstacle problem for nonlinear integro-differential operators}, Calc. Var.
	Partial Differential Equations \textbf{55} (2016), no.~3, Art. 63, 29.
	
	\bibitem{kuusi2}
	Tuomo Kuusi, Giuseppe Mingione, and Yannick Sire, \emph{Nonlocal self-improving
		properties}, Anal. PDE \textbf{8} (2015), no.~1, 57--114.
	
	\bibitem{marce2}
	Paolo Marcellini, \emph{Regularity and existence of solutions of elliptic
		equations with {$p,q$}-growth conditions}, Journal of Differential Equations
	\textbf{90} (1991), no.~1, 1--30.
	
	\bibitem{MSY}
	Tadele Mengesha, Armin Schikorra, and Sasikarn Yeepo, \emph{Calderon-zygmund
		type estimates for nonlocal pde with h{\"o}lder continuous kernel}, Advances
	in Mathematics \textbf{383} (2021), 107692.
	
	\bibitem{NHH}
	Simon Nowak, \emph{Higher {H}\"{o}lder regularity for nonlocal equations with
		irregular kernel}, Calc. Var. Partial Differential Equations \textbf{60}
	(2021), no.~1, Paper No. 24, 37.
	
	\bibitem{NV}
	\bysame, \emph{Regularity theory for nonlocal equations with {VMO}
		coefficients}, arXiv preprint arXiv:2101.11690 (2021), 1--59.
	
	\bibitem{Ni}
	\bysame, \emph{Improved Sobolev regularity for linear nonlocal equations with
		{VMO} coefficients}, Mathematische Annalen (2022), 1--56.
	
	\bibitem{Ph}
	Dian~K. Palagachev, \emph{Global {H}\"{o}lder continuity of weak solutions to
		quasilinear divergence form elliptic equations}, J. Math. Anal. Appl.
	\textbf{359} (2009), no.~1, 159--167.
	
	\bibitem{Scf}
	Armin Schikorra, \emph{Nonlinear commutators for the fractional
		{$p$}-{L}aplacian and applications}, Math. Ann. \textbf{366} (2016), no.~1-2,
	695--720.
	
	\bibitem{SM2}
	James~M. Scott and Tadele Mengesha, \emph{A note on estimates of level sets and
		their role in demonstrating regularity of solutions to nonlocal double phase
		equations}, arXiv preprint
	arXiv: 2011.12779 (2021), pp. 26.
	
	\bibitem{SM}
	James~M. Scott and Tadele Mengesha, \emph{Self-improving inequalities for
		bounded weak solutions to nonlocal double phase equations}, Communications on
	Pure and Applied Analysis \textbf{21} (2022), no.~1, 183--212.
	
	\bibitem{showalter}
	Ralph~E. Showalter, \emph{Monotone operators in {B}anach space and nonlinear
		partial differential equations}, Mathematical Surveys and Monographs,
	vol.~49, American Mathematical Society, Providence, RI, 1997.
	
	\bibitem{silvestre}
	Luis Silvestre, \emph{H\"{o}lder estimates for solutions of
		integro-differential equations like the fractional {L}aplace}, Indiana Univ.
	Math. J. \textbf{55} (2006), no.~3, 1155--1174.
	
	\bibitem{zhikov1}
	Vasili\v{i}~V. Zhikov, \emph{Averaging of functionals of the calculus of
		variations and elasticity theory}, Izv. Akad. Nauk SSSR Ser. Mat. \textbf{50}
	(1986), no.~4, 675--710, 877.
	
\end{thebibliography}
\end{document}